\documentclass[12pt]{article}             

\usepackage{amsfonts}
\usepackage{amsmath}
\usepackage{graphicx}
\usepackage{theorem}                     
\usepackage{color}
\usepackage{hyperref}
  
\def\Section{\setcounter{equation}{0}\section}
\newtheorem{theorem}{Theorem}[section]
\newtheorem{lemma}[theorem]{Lemma}

\newtheorem{definition}[theorem]{Definition}
\newtheorem{proposition}[theorem]{Proposition}
\newtheorem{remark}[theorem]{Remark}

\def\thetheorem{\thesection.\arabic{theorem}}
\def\thesection{\arabic{section}}

\def\theequation {\thesection.\arabic{equation}}
\newenvironment{proof}{\noindent{\bf Proof.~}}
{{\mbox{}\hfill {\small \fbox{}}\\}}
\newenvironment{keywords}{\noindent{\bf Keywords.~}}{}
\newenvironment{AMS}{\noindent{\bf 2010 AMS classifications.~}}{}
\def\beq{\begin{equation}\displaystyle}
\def\eeq{\end{equation}}
\def\bel{\begin{equation} \displaystyle \begin{array}{l} }
\def\eel{\end{array} \end{equation} }
\def\bell{\begin{equation} \displaystyle \begin{array}{ll}  }
\def\eell{\end{array} \end{equation} }

\def\bea{\begin{eqnarray}}
\def\eea{\end{eqnarray} }
\def\bean{\begin{eqnarray*}}
\def\eean{\end{eqnarray*} }
\catcode`@=11
\renewcommand\appendix{\bigskip {\noindent \Large \bf Appendix}
  \setcounter{section}{0}%
  \setcounter{subsection}{0}%
\setcounter{equation}{0}%
\setcounter{theorem}{0}%
\def\thetheorem{A.\arabic{theorem}}
\def\theequation {A.\arabic{equation}}}
\catcode`@=12

\def\eqref#1{(\ref{#1})}

\def\NN{\mathbb{N}}

\def\RR{\mathbb{R}}

\def\bbm1{\mathbbm{1}}

\def\ds{\displaystyle}

\def\eps{\varepsilon}

\def\pa{\partial}

\def\calF{{\cal F}}
\def\calE{{\cal E}}

\def\calM{{\cal M}}

\def\calP{{\cal P}}
\def\calD{{\cal D}}

\def\smes{{\cal S}_{\cal M}}

\def\achapo{\widehat{a}}

\title{Numerical methods for one-dimensional aggregation equations}

\author{Francois James\thanks{Math\'ematiques -- Analyse, Probabilit\'es, Mod\'elisation -- Orl\'eans (MAPMO),
Universit\'e d'Orl\'eans \& CNRS UMR 7349,
F\'ed\'eration Denis Poisson, Universit\'e d'Orl\'eans \& CNRS FR 2964,
45067 Orl\'eans Cedex 2, France ({\tt francois.james@univ-orleans.fr})}
        \and Nicolas Vauchelet\thanks{UPMC Univ Paris 06, UMR 7598, Laboratoire Jacques-Louis Lions,
CNRS, UMR 7598, Laboratoire Jacques-Louis Lions and
INRIA Paris-Rocquencourt, EPI MAMBA,
F-75005, Paris, France ({\tt vauchelet@ann.jussieu.fr})}
\thanks{The second author is supported by the French ``ANR blanc'' project 
Kibord: ANR-13-BS01-0004.}}

\begin{document}

\maketitle

\begin{abstract}
We focus in this work on the numerical discretization of the one dimensional
aggregation equation $\pa_t\rho + \pa_x (v\rho)=0$, $v=a(W'*\rho)$,
in the attractive case.
Finite time blow up of smooth initial data 
occurs for potential $W$ having a Lipschitz singularity at the origin.
A numerical discretization is proposed for which the convergence 
towards duality solutions of the aggregation equation is proved.
It relies on a careful choice of the discretized macroscopic velocity $v$ 
in order to give a sense to the product $v \rho$.
Moreover, using the same idea, we propose an asymptotic preserving scheme 
for a kinetic system in hyperbolic scaling converging towards the aggregation
equation in hydrodynamical limit.
Finally numerical simulations are provided to illustrate the results.
\end{abstract}



\Section{Introduction}

This paper is devoted to the numerical approximation of the so-called 
aggregation equation which writes in one space dimension
	\begin{equation}\label{eq:aggreg}
\pa_t\rho + \pa_x \big(a(W' * \rho)\rho\big)=0.
	\end{equation}
This equation is complemented with some initial data $\rho(0,x)=\rho^{ini}(x)$.
This nonlocal and nonlinear conservation equation 
is involved in many applications in physics and biology
(see e.g. \cite{benedetto,burger,li,morale,okubo,topaz} in the case $a$ linear.
It describes the behaviour of a population of particles (in physical applications) 
or cells (in biological applications) interacting under a continuous interaction potential $W$.
The quantity $\rho$ denotes the density of these particles or cells.
The function $a$ is often linear ($a(u)=\pm u$), but in several applications, such as pedestrian motion 
\cite{pieton,pieton2} 
or chemotaxis (see \cite{parma} and Section \ref{modelchemo} below) a specific nonlinearity has to be considered. 
Depending on the choice of the potential $W$ and the function $a$, one can be in the repulsive or in the attractive case,
which leads to aggregation phenomena. 

In this work we mainly focus on the case involving attractive forces.
Individuals attract one another under the action of the potential $W$, assumed to be 
smooth away of $0$ and bounded from below. 
More precisely, $W$ satisfies the following properties:
	\begin{definition}\label{assump}
We say that $W\in C^1(\RR\setminus\{0\})$ is a pointy attractive potential if
	\beq\label{hyp}
W''= -\delta_0 + w, \quad w\in C_b(\RR),
\quad \mbox{ with }\ \|w\|_{L^1(\RR)} = w_0 <\infty,
	\eeq
 in the distributional sense, where $\delta_0$ is the Dirac measure at $0$.
	\end{definition}

Attractivity in the nonlinear case is ensured provided the function $a$ satisfies
	\begin{equation}\label{hyp_a}
a\in C^1(\RR),\qquad 0\leq a'(x) \leq \alpha, \quad \forall\, x \in \RR.
	\end{equation}

This case has been extensively studied in the linear case $a=\mbox{id}$ \cite{Bertozzi1, Bertozzi2, Bertozzi3}
and it is known that if the potential $W$ has a Lipschitz singularity
then weak solutions blow up in finite time (see e.g. \cite{Bertozzi1,parma}), so that measure valued solutions arise. 
At the theoretical level, global in time existence has been obtained in the linear case $a=\mbox{id}$ 
and in any space dimension by Carrillo et al. \cite{Carrillo}, in the space 
${\mathcal P}_2$ of probability measures with finite second moment,
using the geometrical approach of gradient flows. 
In the nonlinear case, but in one space dimension, global existence of measure solutions has been obtained
by completely different means in \cite{jamesnv2}, namely thanks to the notion of duality solutions.
It has also been proved in \cite{jamesnv2} that in the linear case $a=\mbox{id}$, both notions coincide. 
The key point leading to both uniqueness of solutions and equivalence between the two notions 
is the definition of the macroscopic velocity. In the framework of
gradient flows, it is defined as the unique element with minimal norm in the subdifferential of the
interaction energy associated to $W$, where ${\mathcal P}_2$ is endowed with the Wasserstein distance
(see \cite{Ambrosio,Carrillo} for more details). 
In \cite{jamesnv2}, the macroscopic velocity is defined using the chain rule for BV functions, and this is the viewpoint
adopted for numerical analysis.

In some applications, the aggregation equation is the 
hydrodynamic limit of some kinetic sytem (see \cite{dolschmeis,filblaurpert,parma} for examples in chemotaxis).
We consider here the following kinetic model with relaxation
in hyperbolic scaling
	\begin{equation}\label{eq:kinintro}
\pa_tf_\eps + v\pa_x f_\eps = \frac 1\eps ( \rho_\eps E(v,W'*\rho_\eps) - f_\eps),  
	\end{equation}
where $f_\eps$ is the ditribution function of cells at time $t$, position $x$ and velocity $v$, 
$\rho_\eps=\int f_\eps\,dv$ and the equilibrium function $E$ is normalized so that $\int E\,dv=1$.
Existence of global in time $L^\infty$ weak solutions for such a kinetic
equation with fixed $\eps>0$ is well-known (see e.g. \cite{calvez,nv}).
However taking the limit $\eps \to 0$, we recover the aggregation equation
(\ref{eq:aggreg}) (see Theorem \ref{ConvKinet} below),
for which solutions blow up in finite time. 
Then an interesting issue consists in providing a numerical scheme
for the kinetic system \eqref{eq:kinintro} which allows to 
recover the asymptotic limit when $\eps \to 0$.
Such schemes are usually called {\it asymptotic preserving (AP)} \cite{jin2}.
They are of great interest for kinetic equations since letting $\eps\to 0$
with the mesh size and time step fixed, the scheme becomes a scheme 
for the macroscopic limit (see e.g. \cite{filbetjin,pareschi1,lemou}).
In other words, AP schemes allow a numerical discretization whose time step 
is not constrained by some constant depending on $\eps$.
We refer to \cite{jin} for a review on AP schemes.

The aim of this work is precisely to design numerical methods for (\ref{eq:aggreg}) and \eqref{eq:kinintro}
that are able to capture the measure solutions after blow-up.
The main difficulty is that after blow up the velocity $a(W'*\rho)$ is discontinuous, so that the definition
of the flux has to be considered with great care. Following the principle that holds at the continuous level,
the numerical velocity is obtained thanks a careful discretization
of the Vol'pert calculus for BV functions. We emphasize that the numerical solution may depend upon the 
way of discretizing the velocity. For equation \eqref{eq:aggreg} we work directly on the definition of
$a(W'*\rho)$, for the kinetic model, the discretization is defined through the right-hand
side of equation \eqref{eq:kinintro}. The final scheme is obtained by a splitting technique,
as in for instance \cite{pareschi1,pareschi2}, which is in this particular case very easy to implement.
A more sophisticated technique consists in using well-balanced schemes \cite{gosse3}
as it has been successfully used for chemotaxis models in 
\cite{gosse1,gosse2}. 
However, it is not clear that such schemes allow to recover the solutions after blow up.

The outline of the paper is as follows. Section \ref{sec:aggreg} is devoted
to the aggregation equation \eqref{eq:aggreg}. After recalling existence
and uniqueness result for this system, we provide a numerical scheme and 
prove its convergence.
In Section \ref{APsec}, we consider the kinetic equation \eqref{eq:kinintro}.
We first establish the rigorous derivation of the aggregation equation 
thanks to a hyperbolic limit $\eps \to 0$. Then we propose an asymptotic 
preserving scheme and prove its convergence.
Finally, Section \ref{simul} is devoted to some numerical simulations.

Part of these results were announced in \cite{hyp_proc}.

\Section{Aggregation equation}\label{sec:aggreg}
\subsection{Existence of duality solutions}

For $Y$ and $Z$ two metric spaces we denote $C_b(Y,Z)$ the set
of continuous and bounded functions from $Y$ to $Z$, 
$C_0(Y,Z)$ the set of those that vanish at infinity 
and $C_c(Y,Z)$ the set of continuous functions with compact support.
Let $\calM_b(\RR)$ be the set of bounded Radon measures
and by $\calP_1(\RR)$ the set of positive measure in $\calM_b(\RR)$
such that $\int_\RR |x|d\mu(x)<\infty$.
From now on, the space $\calM_b(\RR^N)$ is
always endowed with the weak topology $\sigma({\cal M}_b,C_0)$.
We denote $\smes :=C([0,T];{\cal M}_b(\RR^N)-\sigma({\cal M}_b,C_0))$.

Duality solutions have been introduced in \cite{bj1} to solve
scalar conservation laws with discontinuous coefficients. More precisely, 
it gives sense to measure valued solutions of the scalar conservation
law
	$$
\pa_t \rho(t,x) + \pa_x(b(t,x)\rho(t,x)) = 0,
	$$
where $b\in L^\infty((0,T)\times \RR)$ satisfies the so-called
one-sided Lipschitz condition 
\begin{equation}\label{OSLC}
\partial_x b(t,.)\leq \beta(t)\qquad\mbox{for $\beta\in L^1(0,T)$,
in the distributional sense}.
\end{equation}
This key point suggests that the velocity field should be compressive.
We refer to \cite{bj1} for the precise definition and general properties
of these solutions. 

Let us first define a notion of duality solution for the aggreagation equation \eqref{eq:aggreg}
in the spirit of \cite{BJpg,jamesnv}:
	\begin{definition}\label{defexist}
We say that $\rho\in C([0,T];\calM_b(\RR))$
is a duality solution to \eqref{eq:aggreg} if there exists 
$\widehat{a}_\rho\in L^\infty((0,T)\times\RR)$ and $\alpha\in L^1_{loc}(0,T)$ 
satisfying $\pa_x\widehat{a}_\rho\le\alpha$ in $\calD'((0,T)\times\RR)$, 
such that for all $0<t_1<t_2<T$,
	$$
\pa_t\rho + \pa_x(\widehat{a}_\rho\rho) = 0\qquad\mbox{in the sense of duality on }(t_1,t_2),
	$$
and $\widehat{a}_\rho=a(W'*\rho)$ a.e.
	\end{definition}

From now on, we denote by $A$ the antiderivative of $a$ such that $A(0)=0$.
Using the chain rule, a natural definition of the flux is
	\beq\label{DefFluxJ}
J:= -\pa_xA(W'*\rho)+a(W'*\rho)w*\rho.
	\eeq
In fact, a formal computation shows that
	$$
-\pa_xA(W'*\rho)=-a(W'*\rho)W''*\rho = a(W'*\rho)(\rho-w*\rho),
	$$
where we use \eqref{hyp} for the last identity.

Then we are in position to state the existence and uniqueness result 
of \cite{jamesnv2}.
	\begin{theorem}[\cite{jamesnv2}, Theorem 3.9]
\label{ExistFlux}
Let us assume that $\rho^{ini}$ is given in $\calP_1(\RR)$.
Under Assumptions \ref{assump} on the potential $W$ and \eqref{hyp_a}
for the nonlinear function $a$, for all $T> 0$ there exists
a unique duality solution $\rho$ of \eqref{eq:aggreg} in the sense 
of Definition \ref{defexist} with $\rho\geq 0$, $\rho(t)\in \calP_1(\RR)$
for $t\in (0,T)$ and which satisfies in the distributional sense:
	\beq\label{eqrhodis}
\pa_t \rho + \pa_x J = 0,
	\eeq
where $J$ is defined in \eqref{DefFluxJ}.
Moreover, there exists $\widehat{a}$, called universal representative, 
such that $\widehat{a}=a(W'*\rho)$ a.e. 
Then $\rho= X_\#\rho^{ini}$, where $X$ is the Filippov flow associated 
to the velocity $\widehat{a}$.
	\end{theorem}

\subsection{Numerical discretization}
Let us consider a uniform space discretization with step $\delta x$ and
denote by $\delta t$ the time step; 
then $t_n=n \delta t$ and $x_i=x_0+i\delta x$, $i=0,\ldots,N_x$. 
We assume that the initial datum $\rho^{ini}$ is
compactly supported with support included in $[x_0,x_{N_x}]$,
and since this work is not concerned with boundary conditions, we assume as well
that the solutions are compactly supported 
in the computational domain during the computational time. 
Then from now on, we take $\rho_0^n=S_0^n=S_1^n=J_{-1/2}^n=0$ and
$\rho_{N_x}^n=S_{N_x}^n=J_{N_x+1/2}^n=0$.

For $n\in \NN$, we assume to have computed an approximation 
$(\rho_i^n)_{i=0,\ldots,N_x}$ of $(\rho(t_n,x_i))_{i=0,\ldots,N_x}$, 
we denote by $(S_i^n)_{i=0,\ldots,N_x}$ an approximation of 
$(W*\rho(t_n,x_i))_{i=0,\ldots,N_x}$ and by $(\nu_i^n)_{i=0,\ldots,N_x}$
an approximation of $(w*\rho(t_n,x_i))_{i=0,\ldots,N_x}$.
Let us denote $\lambda = \delta t/\delta x$ and 
$M$ the total mass of the system, $M= |\rho^{ini}|(\RR)$.
We obtain an approximation of $\rho(t_{n+1},x_i)$ denoted $\rho_i^{n+1}$
by using the following Lax-Friedrichs discretization of equation 
(\ref{eqrhodis})--(\ref{DefFluxJ}):
	\begin{align}
\label{eq:rhoprime}
\rho_i^{n+1} = \rho_i^n - \frac{\lambda}{2} (J_{i+1/2}^n -J_{i-1/2}^n)
	+ \frac{\lambda}{2} c (\rho_{i+1}^n - 2\rho_i^n + \rho_{i+1}^n) \\
\label{eq:Jdis}
J_{i+1/2}^n = -\frac{A(\pa_xS_{i+1}^n)-A(\pa_xS_i^n)}{\delta x} +
	a_{i+1/2}^n \frac{\nu_{i+1}^n+\nu_i^n}{2},
	\end{align}
where we have defined 
	\beq\label{defc}
c:= \max_{x\in [-M(1+w_0),M(1+w_0)]}|a(x)|.
	\eeq
In this scheme, we use the discretization 
	\begin{equation}\label{paW}
\pa_xS_{i+1}^n=\frac{S_{i+2}^n-S_i^n}{2\delta x},
	\end{equation}
and the approximation
	\begin{equation}\label{eq:ai12}
  a_{i+1/2}^n = \left\{\begin{array}{ll}
 0 &  \mbox{ if }\ \pa_xS_{i+1}^n= \pa_xS_i^n,  \\[2mm]
\displaystyle \frac{A(\pa_xS_{i+1}^n)-A(\pa_xS_i^n)}{\pa_xS_{i+1}^n-\pa_xS_i^n} \quad & \mbox{ otherwise.} 
\end{array}\right.
	\end{equation}
We need now a scheme for $S_i^n$. From Assumption 
\ref{assump}, we deduce by taking the convolution of \eqref{hyp} with
$\rho$ that $-W''*\rho + w*\rho = \rho$.
This equation is discretized by using a standard finite difference
scheme:
	\begin{equation}\label{eq:Vwdis}
-\frac{S_{i+1}^n-2S_i^n+S_{i-1}^n}{\delta x^2}+\nu_i^n = \rho_i^n.
	\end{equation}
This scheme allows the computation of $(S_i^n)_i$, provided $(\nu_i^n)_i$ is 
known. For the computation of $(\nu_i^n)_i$, there are multiple ways; 
here we propose to use a piecewise constant
approximation for $\rho$ on each interval $[x_i,x_{i+1})$, so that
	$$
\nu_i^n = \sum_{k=1}^{N_x} \int_{x_k}^{x_{k+1}} \rho_i^n w(x_i-y) \,dy,
	$$
which can be rewritten
\beq\label{eq:nu}
\nu_i^n = \sum_{k=1}^{N_x} \rho_k^n w_{ki},\quad w_{ki}=\int_{x_i-x_{k+1}}^{x_i-x_k} w(z)\,dz= \int_{(i-1-k)\delta x}^{(i-k)\delta x} w(z)\,dz.
\eeq

Now the final version of the scheme for $\rho_i^n$ can be written. Using (\ref{paW}) and (\ref{eq:Vwdis}), 
we deduce that we can rewrite (\ref{eq:Jdis}) as
	\begin{equation}\label{eqJa}
J_{i+1/2}^n = a_{i+1/2}^n \frac{\rho_{i+1}^n+\rho_i^n}{2}.
	\end{equation}
Injecting this latter expression of the flux in (\ref{eq:rhoprime}) we obtain
	\begin{equation}\label{eq:rhoprimebis}
\rho^{n+1}_i = \rho_i^n\big(1-\lambda c+\frac\lambda 4(a_{i-1/2}^n-a_{i+1/2}^n)\big)
+\frac{\lambda}{2}\left(c+\frac{a_{i-1/2}^n}{2}\right)\rho_{i-1}^n 
+\frac{\lambda}{2}\left(c-\frac{a_{i+1/2}^n}{2}\right)\rho_{i+1}^n.
	\end{equation}
We emphasize at this point the importance of the choice of the discretization
of the macroscopic velocity $a_{i+1/2}^n$ in \eqref{eq:ai12} as it is the one
corresponding to the universal representative $\achapo$ of Theorem \ref{ExistFlux}.
Numerical example showing a wrong dynamics with 
a different discretization choice will be provided in Section \ref{modelchemo}.

\begin{remark}
The choice of the discretization \eqref{eq:ai12} 
for the macroscopic velocity can be seen as a 
consequence of the chain rule (or Vol'pert calculus)
for $BV$ functions \cite{volpert} (see also remark 3.98 of \cite{ambrosioBV}):
for a BV function $u$, the fonction $\achapo_V$ defining the chain rule 
$\pa_xA(u)=\achapo_V \pa_xu$ is constructed by
	\begin{equation}\label{eq:volpert1}
\achapo_V(x) = \int_0^1 a(tu_1(x)+ (1-t)u_2(x))\,dt, 
	\end{equation}
where 
	\begin{equation}\label{eq:volpert2}
(u_1,u_2)=\left\{\begin{array}{ll}
\ds (u,u) & \qquad \mbox{ if } x\in \RR\setminus S_u, \\[2mm]
\ds (u^+,u^-) & \qquad \mbox{ if } x\in J_u, \\[2mm]
\ds \mbox{ arbitrary } & \qquad \mbox{elsewhere.}
\end{array}\right.
	\end{equation}
We have denoted by $S_u$ the set of $x\in \RR$ where $u$ does not admit an approximate limit and by $J_u\subset S_u$ the set of jump points.
Applying that to $u=\pa_xS$, we obtain,
	\begin{equation}\label{aVolpert}
\achapo_V(x) = \left\{\begin{array}{ll}
\ds a(\pa_xS(x)) & \mbox{ if }x\in \RR\setminus S_u, \\[2mm]
\ds \frac{A(\pa_xS(x^+))-A(\pa_xS(x^-))}{\pa_xS(x^+)-\pa_xS(x^-)}& \mbox{ if } x\in J_u, \\[3mm]
\ds \mbox{ arbitrary } & \mbox{elsewhere.}
	\end{array}\right.
	\end{equation}
\end{remark}

\subsection{Numerical analysis}
In this subsection, we prove the convergence of the numerical scheme 
defined in \eqref{eq:rhoprime}--\eqref{eq:ai12} towards the unique 
duality solution of Theorem \ref{ExistFlux}.
We first state a Lemma which proves a CFL-like condition for the scheme:
	\begin{lemma}\label{lem:rhopos}
Let us assume that (\ref{hyp_a}) holds and that the condition
	\begin{equation}\label{CFL}
\lambda:= \frac{\delta t}{\delta x} \leq \frac{2}{3c},
	\end{equation}
is satisfied with $c$ defined in \eqref{defc}.
Let us assume that $\rho^{ini}\in \calP_1(\RR)$ is given, compactly
supported and nonnegative, and we define 
$\rho_i^0=\frac{1}{\delta x}\int_{x_i}^{x_{i+1}} \rho^{ini}(dx)\geq 0$.

Then for all $i$ and $n\in \NN$, the sequences computed thanks
to the scheme defined in (\ref{eq:rhoprime})--(\ref{eq:Vwdis}) satisfy
	$$
\rho_i^n \geq 0, \qquad |a_{i+1/2}^n|\leq c.
	$$
	\end{lemma}
\begin{proof}
We choose $x_0$ and $x_{N_x}$ such that supp$(\rho^{ini})\subset [x_0,x_{N_x}]$.
Let us define $M_i^n=\delta x\sum_{j=0}^i \rho_j^n$ and 
$M^{n+1}_i=\delta x\sum_{j=0}^i \rho^{n+1}_j$.
Since the scheme (\ref{eq:rhoprime}) is conservative, 
we have $M_{N_x}^n=M_{N_x}^0=M$.
Clearly, $\rho_i^n=(M_i^n-M_{i-1}^n)/\delta x$ and from (\ref{eqJa}) we have
$J_{i+1/2}^n=a_{i+1/2}^n(M_{i+1}^n-M_{i-1}^n)/(2\delta x)$.
Then we deduce from (\ref{eq:rhoprime}) that
\begin{equation}\label{eq:Mprime}
M_i^{n+1} = (1-\lambda c)M_i^n + \frac{\lambda}{2}\left(c-\frac{a_{i+1/2}^n}{2}\right)M_{i-1}^n
+ \frac{\lambda}{2}\left(c+\frac{a_{i+1/2}^n}{2}\right)M_{i+1}^n.
\end{equation}

By definition of $\nu_i^n$ in \eqref{eq:nu}, we deduce from \eqref{hyp} that
for all $i\in \NN^*$, 
\beq\label{eq:boundnu}
\delta x \sum_{j=1}^i \big| \nu_j^n \big| \leq \delta x 
\sum_{k=1}^{N_x} \rho_k^n w_0 = M w_0.
\eeq
Moreover, from the definition of $\pa_xS_i^n$ in \eqref{paW} and using equation 
\eqref{eq:Vwdis}, we deduce 
	$$
-\frac{\pa_xS_{i+1}^n-\pa_xS_i^n}{\delta x} 
+ \frac{\nu_{i+1}^n +\nu_i^n}{2} = \frac{\rho_{i+1}^n+\rho_i^n}{2}.
	$$
Summing this latter equation over $i$, we obtain
	$$
\pa_xS_{i+1}^n=\pa_xS_0^n + \frac 12 \big(-M_{i+1}^n-M_i^n+M_0^n +
\delta x \big(2\sum_{j=0}^i \nu_j^n +\nu_{i+1}^n-\nu_0^n\big) \big).
	$$
Using the boundary conditions, we have $\pa_xS_0^n=0$ and $M_0^n=0$.
Then,
\beq\label{eq:paxSM}
\pa_xS_{i+1}^n = -\frac 12 \big(M_{i+1}^n+M_i^n -
\delta x\big(2\sum_{j=0}^i \nu_j^n +\nu_{i+1}^n-\nu_0^n\big) \big).
\eeq

We are now in position to prove the lemma by an induction on $n$.
For $n=0$, by construction of the initial data, we have $\rho_i^0\geq 0$.
Then for all $i$, we have $0\leq M_i^0\leq M$ and with \eqref{eq:paxSM}
and \eqref{eq:boundnu} we deduce that
	$$
|\pa_xS_{i+1}^0| \leq M(1+w_0), \quad \mbox{ for all } i.
	$$
Futhermore, since we have
	$$
\frac{A(\pa_xS_{i+1}^0)-A(\pa_xS_i^0)}{\pa_xS_{i+1}^0-\pa_xS_i^0} = a(\theta_i^0), 
	\qquad \theta_i^0\in (\pa_xS_i^0,\pa_xS_{i+1}^0)\subset (-M(1+w_0),M(1+w_0)),
	$$
we deduce with \eqref{eq:ai12} that $|a_{i+1/2}^0|\leq c$, which proves the result for $n=0$.

Let us assume that $\rho_i^n\geq 0$ and $|a_{i+1/2}^n|\leq c$, for 
some $n\in \NN$. 
From condition \eqref{CFL} and the induction assumption
$|a_{i+1/2}^n|\leq c$, we deduce that in the scheme \eqref{eq:rhoprimebis},
all the coefficients in front of $\rho_{i-1}^n$, $\rho_i^n$ 
and $\rho_{i+1}^n$ are nonnegative.
Thus $\rho_i^{n+1}\geq 0$ for all $i$.
Moreover, we have clearly by definition that $0\leq M_i^n\leq M$.
Then, from the condition \eqref{CFL} and induction assumption 
$|a_{i+1/2}^n|\leq c$, we deduce with \eqref{eq:Mprime} that 
$M^{n+1}_i$ is a convex combination of $M^n_{i+1}$, $M_i^n$ and $M_{i-1}^n$. 
Then $0\leq M_i^{n+1}\leq M$.
Thus, as above, using \eqref{eq:paxSM} with $n+1$ instead of $n$,
we have $|\pa_xS_{i+1}^{n+1}|\leq M(1+w_0)$, which implies
$|a_{i+1/2}^{n+1}|\leq c$.
\end{proof}

Let us define the reconstruction
	$$
\rho_\delta(t,x) = \sum_{n\in\NN} \sum_{i=0}^{N_x} \rho_i^n {\bf 1}_{[n\delta t,(n+1)\delta t)\times [x_i,x_{i+1})}(t,x),
	$$
and $S_\delta$, $\pa_xS_\delta$, $\nu_\delta$, $J_\delta$ and $a_\delta$ 
are defined in a similar way thanks to $(S_i^n)_i$, $(\pa_xS_i^n)_i$, 
$(\nu_i^n)_i$, $(J_{i+1/2}^n)_i$ and $(a_{i+1/2}^n)_i$ respectively.
Then we have the following convergence result:
\begin{theorem}\label{th:convmacro}
Let us assume that $\rho^{ini}\in \calP_1(\RR)$ is given,
compactly supported and nonnegative
and define $\rho_i^0=\frac{1}{\delta x}\int_{x_i}^{x_{i+1}} \rho^{ini}(dx)\geq 0$.
Under assumption (\ref{hyp_a}), if (\ref{CFL}) is satisfied, 
then the discretization $\rho_\delta$ converges in $\smes$ towards the 
unique duality solution $\rho$ of Theorem \ref{ExistFlux} 
as $\delta t$ and $\delta x$ go to $0$.
\end{theorem}
\begin{proof}
As above, we assume that supp$(\rho^{ini})\subset[x_0,x_{N_x}]$.
Applying Lemma \ref{lem:rhopos}, we have that 
$(\rho_i^n)_i$ is nonnegative and $|a_{i+1/2}^n|\leq c$, 
provided \eqref{CFL} is satisfied.
As in the proof of Lemma \ref{lem:rhopos}, we define 
$M_i^n=\delta x\sum_{j=0}^i \rho_j^n$, which satisfies \eqref{eq:Mprime}
and $0\leq M_i^n \leq M$.
Moreover, we clearly have that $0\leq \rho_i^n=(M_i^n-M_{i-1}^n)/\delta x$, 
then equation (\ref{eq:rhoprimebis}) implies a $BV(\RR)$ estimate on 
$(M_i^n)_i$, provided \eqref{CFL} is satisfied.
More precisely the scheme is TVD for the sequence $(M_i^n)_i$.

Defining
	$$
M_\delta(t,x) = \sum_{n\in\NN} \sum_{i=0}^{N_x} M_i^n 
{\bf 1}_{[n\delta t,(n+1)\delta t)\times [x_i,x_{i+1})}(t,x),
	$$
we deduce from standard techniques that we have a 
$L^\infty\cap BV((0,T)\times\RR)$ estimate on $M_\delta$. 
It implies the convergence, up to a subsequence, 
of $M_\delta$ in $L^1_{loc}(\RR^+\times \RR)$ towards a function 
$\widetilde{M}\in L^\infty\cap BV((0,T)\times\RR)$ 
when $\delta t$ and $\delta x$ go to $0$ and satisfy (\ref{CFL}).

Let us define $\rho = \pa_x\widetilde{M} \in L^\infty((0,T);\calM_b(\RR))$. 
Obviously, noting that $\rho_i^n=(M_i^n-M_{i-1}^n)/\delta x$, we deduce that
$\rho$ is the limit in $\smes$ of $\rho_\delta$.
By definition \eqref{eq:nu}, we have that $\nu_\delta = w*\rho_\delta$.
Therefore, the sequence $(\nu_\delta)_\delta$ converges,
up to a subsequence, towards $\nu:=w*\rho$ for a.e. $t>0$ and $x\in \RR$.

From \eqref{eq:paxSM}, we deduce that we have the same bound on the sequence 
$(\pa_xS_i^n)_{i,n}$ as on $(M_i^n)_{i,n}$.
We conclude that the sequence $(\pa_xS_i^n)_{i,n}$ is bounded in
$L^\infty\cap BV((0,T)\times\RR)$. 
As above, we get the convergence, up to a subsequence,
in $L^1_{loc}(\RR^+\times \RR)$ of $\pa_xS_\delta$ towards a function 
$\pa_xS$ belonging to $L^\infty\cap BV((0,T)\times\RR)$
as $\delta t$ and $\delta x$ go to $0$ and satisfy \eqref{CFL}.
By definition of $\pa_xS_\delta$, we have the strong convergence 
up to a subsequence in $L^1_{loc}(\RR^+,W^{1,1}_{loc}(\RR))$ 
of $S_\delta$ towards $S$.

Passing to the limit in the equation \eqref{eq:Vwdis} we deduce that 
$S$ and $w$ satisfy in the weak sense the equation
	$$
-\pa_{xx} S + w = \rho.
	$$
Moreover, from Lemma \ref{lem:rhopos},
we deduce that the sequence $(a_\delta)_\delta$ is bounded in $L^\infty$,
thus we can extract a subsequence converging
in $L^\infty -weak^*$ towards $\widetilde{a}$.
From the $L^1_{loc}$ convergence of $(\pa_xS_\delta)_\delta$, we deduce that 
$\widetilde{a}=a(\pa_xS)$ a.e.
Then, from (\ref{eq:Jdis}), we have the convergence in the sense of 
distribution of $J_\delta$ towards 
$J=-\pa_x(A(\pa_xS)) + \widetilde{a} w$ a.e.
Finally, taking the limit in the distributional sense of equation 
(\ref{eq:rhoprime}) we deduce that $\rho$ is a solution in
the sense of distribution of (\ref{eqrhodis})--(\ref{DefFluxJ}).
By uniqueness of this solution, we deduce that $\rho$ is the unique 
duality solution of Theorem \ref{ExistFlux}.
\end{proof}

Finally, we notice that, as in the continuous case (see \cite{jamesnv}), 
the nonnegativity of the density $\rho$ allows to ensure an OSL condition
on the discretized macroscopic velocity.
\begin{proposition}\label{adisOSL}
With the same notations and assumptions as in Theorem \ref{th:convmacro},
the discrete macroscopic velocity in \eqref{eq:ai12} 
satisfies the discrete OSL condition:
	$$
\frac{1}{\delta x} \big(a_{i+1/2}^n - a_{i-1/2}^n\big) \leq C \alpha.
	$$
\end{proposition}

\begin{proof}
From definition \eqref{eq:ai12} we have, applying the mean value Theorem:
	$$
\frac{1}{\delta x} \big(a_{i+1/2}^n - a_{i-1/2}^n\big)=
\frac{1}{\delta x} \big(a(\theta_{i+1/2}^n) - a(\theta_{i-1/2}^n)\big)=
\frac{a'(\gamma_i^n)}{\delta x} \big(\theta_{i+1/2}^n - \theta_{i-1/2}^n\big),
	$$
where $\theta_{i+1/2}^n\in (\pa_xS_i^n,\pa_xS_{i+1}^n)$ and 
$\gamma_i^n\in (\theta_{i-1/2}^n,\theta_{i+1/2}^n)$ (where the interval $(\alpha,\beta)$ 
is the interval $(\beta,\alpha)$ when $\beta<\alpha$).
Then, using assumption \eqref{hyp_a},
	$$
\frac{1}{\delta x} \big(a_{i+1/2}^n - a_{i-1/2}^n\big) \leq
\frac{a'(\gamma_i^n)}{\delta x} \max \big\{0,\pa_xS_{i+1}^n-\pa_xS_i^n,
\pa_xS_i^n-\pa_xS_{i-1}^n,\pa_xS_{i+1}^n-\pa_xS_{i-1}^n\big\}.
	$$
From the definition \eqref{paW} and with \eqref{eq:Vwdis}, we have
	$$
\frac{\pa_xS_{i+1}^n-\pa_xS_i^n}{\delta x}
=\frac{\nu_{i+2}^n+\nu_i^n}{2}-\frac{\rho_{i+1}^n+\rho_i^n}{2} \leq
\frac{\nu_{i+2}^n+\nu_i^n}{2},
	$$
and
	$$
\frac{\pa_xS_{i+1}^n-\pa_xS_{i-1}^n}{\delta x}
	=\frac{\nu_{i+2}^n+2\nu_i^n+\nu_{i-1}^n}{2}-\frac{\rho_{i+1}^n+2\rho_i^n+\rho_{i-1}^n}{2}
	\leq \frac{\nu_{i+2}^n+2\nu_i^n+\nu_{i-1}^n}{2},
	$$
where we use the nonnegativity on the sequence $(\rho_i^n)_{i,n}$.
Since the sequence $(\nu_i^n)_{i,n}$ is bounded in $L^\infty$, we deduce that
$\pa_xS_{i+1}^n-\pa_xS_i^n$ and $\pa_xS_{i+1}^n-\pa_xS_{i-1}^n$ are bounded from 
above by a nonnegative constant $C$ only depending on the initial data. 
Using moreover assumption \eqref{hyp_a} allows to conclude the proof.
\end{proof}

\begin{remark}
In applications, we can have $w=W$. In this case, we prefer to set 
$\nu_i^n=S_i^n$ instead of \eqref{eq:nu}. And the sequence $(S_i^n)_{i,n}$ 
is then entirely determined by solving system \eqref{eq:Vwdis}.
Then it is straightforward to adapt the proof of Theorem \ref{th:convmacro};
the convergence result still holds in this case.
\end{remark}

\Section{Asymptotic preserving scheme}\label{APsec}
In this section, we consider an asymptotic preserving scheme 
allowing to recover the numerical discretization 
\eqref{eq:rhoprime}--\eqref{eq:Jdis} from a kinetic model \eqref{eq:kinintro}.

Asymptotic preserving (AP) schemes has been widely developed 
since the 90s for a wide range of time-dependent kinetic and
hyperbolic equations. The basic idea is to develop a numerical
discretization that preserves the asymptotic limits from the 
microscopic to the macroscopic models \cite{jin}. 
Moreover, in the definition of \cite{jin2}, an AP scheme
should be implemented explicitely (or at least
more efficiently than using a Newton type solvers for nonlinear
algebraic systems).

\subsection{Hydrodynamical limit}
As already mentioned, aggregation equation \eqref{eq:aggreg}
can be derived by a hydrodynamical limit of some kinetic equation.
Here we assume that the kinetic system leading to \eqref{eq:aggreg}
in hyperbolic scaling is given by the following relaxation model
of BGK type
\beq\label{eq:kin}
\pa_t f_\eps + v\pa_x f_\eps = \frac 1\eps \big(\rho_\eps E(v,W'*\rho_\eps) 
-f_\eps\big), \qquad t\geq 0,\ x\in \RR,\ v\in V,
\eeq
where we assume that the equilibrium function $E\geq 0$, $E\in C^2(V\times \RR)$
is normalized such that
\beq\label{hypM1}
\int_V E(v,x)\,dv = 1,\quad \forall\, x \in \RR.
\eeq
Moreover we assume that the domain $V$ is a bounded interval of $\RR$, 
for the clarity of the notations, we will set $V=[-V_M,V_M]$.
We denote by $S_\eps$ the potential $S_\eps=W*\rho_\eps$, 
which, due to \eqref{hyp} is a weak solution to
\beq\label{eq:potS}
-\pa_{xx}S_\eps + w*\rho_\eps = \rho_\eps.
\eeq
We define moreover the antiderivative
	$$
\calE(v,x) = \int^x_0 E(v,y)\,dy,
	$$
so that
	$$
\pa_{xx} S_\eps\, E(v,\pa_xS_\eps) = \pa_x \calE(v,\pa_xS_\eps).
	$$
Then we deduce formally from \eqref{eq:potS} that
	\beq\label{Pieps}
\ds \Pi_\eps:= \rho_\eps E(v,\pa_xS_\eps) = -\pa_x \calE(v,\pa_xS_\eps) + (w*\rho_\eps) E(v,\pa_xS_\eps).
	\eeq

Then the momentum equations resulting from this system are given by
\begin{eqnarray}
\label{eqcin:moment1}
&&\pa_t\rho_\eps + \pa_x J_\eps = 0, \\[1mm]
\label{eqcin:moment2}
&&\pa_t J_\eps+ \pa_x q_\eps = \frac 1\eps \Big(\int_V v\Pi_\eps(v)\,dv -J_\eps\Big),
\end{eqnarray}
where $\rho_\eps = \int_V f_\eps(x,v)\,dv$, $J_\eps=\int_V vf_\eps(x,v)\,dv$
and $q_\eps=\int_V v^2 f_\eps(x,v)\,dv$.
We define 
\beq\label{defaA}
a(x) = \int_V v E(v,x)\,dv \ , \quad \mbox{ and } \quad
A(x) = \int_V v \calE(v,x)\,dv.
\eeq
Obviously, we have $A'=a$. Then, with this notation, we have
	$$
\int_V v\Pi_\eps(v)\,dv = a(\pa_xS_\eps) \rho_\eps = -\pa_x A(\pa_xS_\eps) + (w*\rho_\eps) a(\pa_xS_\eps).
	$$
Letting formally $\eps\to 0$ in \eqref{eqcin:moment2}, we get
	$$
J_\eps \to J_0:=\int_V v\Pi_0(v)\,dv = - \pa_x A(\pa_xS_0) + (w*\rho_0) a(\pa_xS_0).
	$$
Injecting in \eqref{eqcin:moment1}, we recover the aggregation equation
	\beq\label{eqmacro}
\pa_t \rho_0 + \pa_x J_0=0, \qquad J_0 =-\pa_x A(\pa_xS_0) + (w*\rho_0) a(\pa_xS_0),
	\eeq
where $S_0=W*\rho_0$.

We can now establish the rigorous 
derivation of the macroscopic model. This is an extension of the hydrodynamical limit stated in 
Theorem 3.10 of \cite{jamesnv}, 
where a particular case appearing in bacterial chemotaxis is considered.
	\begin{theorem}\label{ConvKinet}
Assume that $V\subset\RR$ is bounded and that assumption \eqref{hyp} holds.
Let $f^{ini}\geq 0$ be given such that 
$\rho^{ini}:= \int_V f^{ini}(v)\,dv$ belongs to $\calP_1(\RR)$.
Let $f_\eps$ be a solution to \eqref{eq:kin}--\eqref{eq:potS}
with $0\leq E \in C^2(V\times\RR)$ satisfying \eqref{hypM1}.
Then, as $\eps \to 0$, $f_\eps$ converges in the following
sense:
	$$
\rho_\eps:=\int_V f_\eps(v)\,dv \rightharpoonup \rho \qquad \mbox{ in } \quad 
\smes := C([0,T];{\cal M}_{b}(\RR)-\sigma({\cal M}_{b},C_0)), 
	$$
where $\rho$ is the unique duality solution of Theorem \ref{ExistFlux}.
	\end{theorem}
\begin{proof}
From \eqref{eqcin:moment1}, we deduce that for all $t\in [0,T]$, 
$|\rho_\eps(t,\cdot)|(\RR)=|\rho^{ini}|(\RR)$. 
Therefore, for all $t\in [0,T]$ 
the sequence $(\rho_\eps(t,\cdot))_\eps$ is relatively
compact in $\calM_b(\RR)-\sigma(\calM_b(\RR),C_0(\RR))$.
Moreover, since the domain $V$ is bounded, we deduce that
$(q_\eps)_\eps$ and $(J_\eps)_\eps$ are bounded in $L^\infty([0,T],L^1(\RR))$
independantly of $\eps$.
Using \eqref{eqcin:moment1}, we deduce the equicontinuity in time
of the sequence $(\rho_\eps)_\eps$. Thus this latter sequence is
relatively compact in $\smes$; up to a subsequence it converges 
towards $\rho$ in $\smes$.

From \eqref{eqcin:moment2}, we have
	\beq\label{conv1}
J_\eps = \int_V v\Pi_\eps(v)\,dv + \eps\big(\pa_tJ_\eps+\pa_x q_\eps\big).
	\eeq
Using assumption \eqref{hyp}, we have 
	$$
\pa_xS_\eps = W'*\rho_\eps = \int_{\RR} H(x-y)\rho_\eps(dy) 
+ \int_\RR \int_{-\infty}^{x-y} w(z)dz \rho_\eps(dy),
	$$
where $H$ is the Heaviside function. Then, using the $L^1$ bound
on $w$ in \eqref{hyp}, we deduce
	$$
|\pa_xS_\eps| \leq (1+w_0)|\rho_\eps|(\RR)= (1+w_0) M.
	$$
Therefore $(\pa_xS_\eps)_\eps$ is bounded in $L^\infty([0,T]\times \RR)$
independantly of $\eps$.

Moreover, by the weak convergence of $(\rho_\eps)_\eps$, we deduce that
$(\pa_xS_\eps)_\eps$ converges weakly towards $\pa_xS=W'*\rho$ 
in $L^\infty\,w-*$ and a.e. (see e.g. Lemma 4.2 of \cite{jamesnv}).
Equivalently we have the strong convergence of $w*\rho_\eps$ towards
$w*\rho$. We deduce that in the distributional sense
\beq\label{conv2}
\int_V v\Pi_\eps(v)\,dv \to J:= -\pa_xA(\pa_xS)+(w*\rho) a(\pa_xS).
\eeq
Combining \eqref{conv1} and \eqref{conv2}, we deduce by taking the
limit in the distributional sense in equation \eqref{eqcin:moment1} that
$$
\pa_t\rho + \pa_x J = 0.
$$
Obviously we have in the distributional sense $-\pa_{xx}S=\rho + w*\rho$.
Finally, thanks to the chain rule for BV function (or Vol'pert calculus), 
there exists $\widehat{a}_V$ such that $\widehat{a}_V=a(\pa_xS)$ a.e. and 
$J=\widehat{a}_V \rho$. By uniqueness of the solution of Theorem \ref{ExistFlux}
we deduce that the solution obtained in the limit $\eps\to 0$ is 
the unique duality solution of this latter Theorem.
\end{proof}

\subsection{An AP numerical scheme}
As above, we consider a time discretization of step $\delta t$, a uniform 
space discretization of step $\delta x$ and a uniform discretization 
of the velocity space of size $\delta v$.
We denote $f_\eps^n$ an approximation of $f_\eps$ at time $t_n=n\delta t$.
The asymptotic preserving scheme we are considering here is based on 
the following time splitting argument:

\noindent$\bullet$ Assuming that approximations $f_\eps^n$, $\rho_\eps^n$ and $S_\eps^n$ 
  of $f_\eps$, $\rho_\eps$ and $S_\eps$ are known at time $t_n$.
  We have now everything at hand to compute $\Pi_\eps^{n}$ from \eqref{Pieps}:
  \beq\label{Piespn}
  \Pi_\eps^{n}(v) := -\pa_x \calE(v,\pa_xS_\eps^{n})+
  (w*\rho_\eps^{n}) E(v,\pa_xS_\eps^{n}).
  \eeq
  We solve in this first step, during a time step $\delta t$,
  the relaxation equation
  \beq\label{eq:step2}
  \pa_t f_\eps = \frac 1\eps (\Pi_\eps - f_\eps),
  \eeq
which allows to compute $f_\eps^{n+1/2}$.
The main point to build an asymptotic preserving scheme 
is that we should recover the good asymptotic when $\eps \to 0$.
We notice by integrating \eqref{eq:step2} on $V$ that $\pa_t \rho_\eps = 0$. 
Then $\rho_\eps^{n+1/2}=\rho_\eps^{n}$ and since \eqref{eq:potS}
depends only on $\rho_\eps$, we have $S_\eps^{n+1/2}=S_\eps^{n}$.
Therefore $\Pi_\eps$ is constant during this time step:
$\Pi_\eps^{n+1/2}=\Pi_\eps^{n}$.
Then we can solve exactly equation \eqref{eq:step2} during this time step by
	\beq\label{eq:f1}
f_\eps^{n+1/2} = e^{-\delta t/\eps} (f_\eps^{n} - \Pi_\eps^{n}) + \Pi_\eps^{n}.
	\eeq

\noindent$\bullet$ In a second step, we discretize during a time step $\delta t$
  the free transport equation:
	$$
\pa_t f_\eps + v \pa_x f_\eps = 0.
	$$
Denoting by $D_x$ some discrete derivative with respect to $x$, we obtain
	\beq\label{eq:f1/2}
f_\eps^{n+1} = f_\eps^{n+1/2} - \delta t v D_x f_\eps^{n+1/2}.
	\eeq
Then we compute $\rho_\eps^{n+1} = \int_V f_\eps^{n+1}(v)\,dv$ and solve 
\eqref{eq:potS} to obtain $S_\eps^{n+1}$. 

This approach allows to construct an order one in time discretization. 
We can construct a second order in time scheme by using the Strang splitting
which consists in solving the first step during a time step $\delta t/2$,
then solving the second step during a time step $\delta t$, finally 
solving again the first step during a time step $\delta t/2$.

With this method, the small parameter $\eps$ is taken into account only 
in the first step.
Letting $\eps\to 0$, we deduce easily from \eqref{eq:f1} that at the limit
$\eps\to 0$, we have
$$
f_\eps^{n+1/2} \to f_0^{n+1/2}= \Pi_0^{n+1/2} = \Pi_0^{n},
$$
since as explained above $\rho_\eps^{n+1/2}=\rho_\eps$ and $S_\eps^{n+1/2}=S_\eps^n$.
Moreover, with \eqref{Piespn} we obtain
$$
\Pi_0^{n}(v) := -\pa_x \calE(v,\pa_xS_0^{n})+
(w*\rho_0^{n}) E(v,\pa_xS_0^{n}).
$$
Then by applying the first step \eqref{eq:f1/2}, we have
$$
f_0^{n+1} = \Pi_0^{n+1/2} + \delta t v D_x \Pi_0^{n+1/2}.
$$
Integrating with respect to $v$, we deduce, using the notations
$\rho_0 = \int_V f_0\,dv$ and $J_0= \int_V v f_0\,dv$ that
$$
\rho_0^{n+1} = \rho_0^{n+1/2} + \delta t D_x J_0^{n+1/2}.
$$
Moreover, we have $\rho_0^{n+1/2}=\rho_0^{n}$ and $J_0^{n+1/2}=J_0^{n}$.
Then,
$$
\rho_0^{n+1} = \rho_0^{n} + \delta t D_x J_0^{n},
$$
which is an explicit in time discretization of the conservation equation
\eqref{eqmacro}.

{\bf Discretization}.
We consider that the velocity space is given by $V=[-V_M,V_M]$ 
and is uniformly discretized by 
$v_j=-V_M+j \delta v$, for $j=0,\ldots,N_v$, and $\delta v = 2V_M /N_v$.
We recall the time and space discretization $t_n=n\delta t$ and
$x_i=x_0+i\delta x$.
As above, $(\rho_i^n)_{i,n}$, $(S_i^n)_{i,n}$ and $(\nu_i^n)_{i,n}$ are 
approximations of resp. $(\rho(t_n,x_i))_{i,n}$, $(S(t_n,x_i))_{i,n}$ and
$(W*\rho(t_n,x_i))$. Moreover, we denote by $(f_{ij}^n)_{i,j,n}$
an approximation of $(f(t_n,x_i,v_j))_{i,j,n}$ and
by $(\Pi_{ij}^n)_{i,j,n}$ an approximation of $(\Pi(t_n,x_i,v_j))_{i,j,n}$
defined in \eqref{Pieps}.

Assuming $(f_{ij}^n)_{i,j,n}$ is known for some $n\in \NN$, we compute 
the approximated density by a trapezoidal rule
	\beq\label{rhoif}
\rho_i^n = \frac{\delta v}{2}(f_{i0}^n+f_{iN_v}^n) + \sum_{j=1}^{N_v-1} \delta v f_{ij}^n.
	\eeq
Then we solve \eqref{eq:Vwdis} where $\nu_i^n$ is obtained
thanks to \eqref{eq:nu}, which allows to compute $S_i^n$.
We introduce moreover the approximation of the function $A$ 
defined in \eqref{defaA} computed by the trapezoidal rule:
	$$
A_\delta(x) := \frac{\delta v}{2}(v_0 \calE(v_0,x) + v_{N_v} \calE(v_{N_v},x))
+ \delta v \sum_{j=1}^{N_v-1} v_j \calE(v_j,x).
	$$
From now on, we denote by $I_\delta$ the linear operator of approximation 
by the trapezoidal rule:
	$$
I_\delta(\calF) = \frac{\delta v}{2}( \calF(v_0) +  \calF(v_{N_v}))
+ \delta v \sum_{j=1}^{N_v-1} \calF(v_j).
	$$
We recall that if $\calF$ is smooth, typically $\calF\in C^2(V)$, then 
we have the error estimate 
	\beq\label{errortrap}
\Big|I_\delta(\calF) - \int_V \calF(v)\,dv \Big| \leq \frac{\delta v^2 V_M\|\calF''\|_\infty}{6}.
	\eeq
Then we have $\rho_i^n=I_\delta((f_{ij}^n)_j)$ and 
$A_\delta(x)=I_\delta((v_j\calE(v_j,x))_j)$.

We have seen in the previous Section that the flux and therefore
the corresponding velocity \eqref{eq:ai12} should be defined with care.
In this aim, we first introduce the following discretization of 
$E(v_j,\pa_xS(t_{n},x_i))$:
	\begin{equation}\label{eq:ui12}
E_{ij}^{n} = \left\{\begin{array}{ll} 0, &  \mbox{ if }\ \pa_xS_{i+1/2}^{n}= \pa_xS_{i-1/2}^{n},  \\[2mm]
\displaystyle \frac{\calE(v_j,\pa_xS_{i+1/2}^{n})-
\calE(v_j,\pa_xS_{i-1/2}^{n})}{\pa_xS_{i+1/2}^{n}-\pa_xS_{i-1/2}^{n}},
\quad & \mbox{ otherwise.} 
	\end{array}\right.
	\end{equation}
In this expression, we use the notation $\pa_xS_{i+1/2}^n=(S_{i+1}^n-S_i^n)/\delta x$.
However, with this discretization of $E$ does not satisfy the normalization condition \eqref{hypM1}.
Therefore we define
	\beq\label{defeij}
e_{ij}^{n} = \left\{\begin{array}{ll}
\ds \frac{E_{ij}^{n}}{I_\delta((E_{ij})_j^{n})}, 
&\qquad \mbox{if } I_\delta((E_{ij})_j^{n}) \neq 0;    \\[5mm]
\qquad \frac{1}{N_v \delta v}, & \qquad \mbox{otherwise,}
	\end{array}\right.
	\eeq
so that $I_\delta((e_{ij})_j^{n})=1$.
Then the quantity $\Pi$ defined in \eqref{Pieps} is approximated by
	\beq\label{PiErho}
\Pi_{ij}^{n} = e_{ij}^{n} \rho_i^{n},
	\eeq
and we have $I_\delta((\Pi_{ij}^{n})_j)=\rho_i^{n}$.
We notice that by using \eqref{eq:Vwdis}, we can rewrite 
	\beq\label{Pidev}
\Pi_{ij}^{n} = -\frac{1}{I_\delta((E_{ij}^{n})_j)}\,
\frac{\calE(v_j,\pa_xS_{i+1/2}^{n})-
\calE(v_j,\pa_xS_{i-1/2}^{n})}{\delta x} + \nu_i^{n} e_{ij}^{n},
	\eeq
which corresponds to a discretization of \eqref{Piespn}.
Finally, we obtain the required approximation of the velocity by multiplying \eqref{eq:ui12} by $v_j$ and summing over $j$:
for each $i$ such that $\pa_xS_{i+1/2}^{n}\neq \pa_xS_{i-1/2}^{n}$ and $I_\delta((E_{ij}^n)_j) \neq 0$, we have
	\beq\label{adelta}
I_\delta((v_j e_{ij}^{n})_j) = \frac{1}{I_\delta((E_{ij})_j^{n})}\,
\frac{A_\delta(\pa_xS_{i+1/2}^{n})-A_\delta(\pa_xS_{i-1/2}^{n})}{\pa_xS_{i+1/2}^{n}-\pa_xS_{i-1/2}^{n}}
:= \widehat{a_\delta}_i^{n}.
	\eeq
For the others integers $i$, we set $\widehat{a_\delta}_i^{n}=0$.

{\bf Numerical scheme.}
Let us assume that $(f_{ij}^{n})_{i,j}$ is known for some $n\in \NN$. 
Then, we compute $\rho_i^{n}=I_\delta((f_{ij}^{n})_j)$ and 
the corresponding macroscopic potential $(S_i^{n})_i$ by solving 
\beq\label{eq:Vwdis2}
-\frac{S_{i+1}^n-2 S_i^n + S_{i-1}^n}{\delta x^2} + \nu_i^n = \rho_i^n,
\eeq
where $(\nu_i^{n})_i$ is obtained as in the previous Section by \eqref{eq:nu}.
We compute $(\Pi_{ij}^{n})_{ij}$ with \eqref{PiErho}--\eqref{defeij}.
Then we have:
\beq\label{fijstep2}
f_{ij}^{n+1/2} = e^{-\delta t/\eps} f_{ij}^{n} + (1-e^{-\delta t/\eps}) 
\Pi_{ij}^{n}.
\eeq
Thanks to our choice of $e_{ij}^n$ in \eqref{defeij}, we have by applying
the operator $I_\delta$ in \eqref{fijstep2} that
$\rho_i^{n+1/2} = \rho_i^{n}$.
Then, we obtain $(f_{ij}^{n+1})_{i,j}$ by, for instance, applying 
a Lax-Friedrichs discretization for the step 2:
\beq\label{fijstep1}
f_{ij}^{n+1} = f_{ij}^{n+1/2} - \frac{\lambda}{2} (v_j f_{i+1,j}^{n+1/2}
-v_j f_{i-1,j}^{n+1/2}) + \frac{\lambda V_M}{2} 
(f_{i+1,j}^{n+1/2}-2f_{ij}^{n+1/2}+f_{i-1,j}^{n+1/2}).
\eeq

\begin{theorem}\label{prop:AP}
Let $V=[-V_M,V_M]$, with $\lambda V_M \leq 1$, and $0\leq E \in C^2(V\times \RR)$
satisfy \eqref{hypM1}. Consider the sequence $(f_{ij}^n)_{i,j,n}$ computed thanks 
to \eqref{eq:Vwdis2}--\eqref{fijstep1}.
Then, as $\eps \to 0$ and $\delta v\to 0$, the sequence
$(\rho_i^n)_{i,n}:= I_\delta ((f_{ij}^n)_j)_{i,n}$ converges weakly, 
up to a subsequence, towards the sequence $(\widetilde{\rho}_i^n)_{i,n}$,
computed by a Lax-Friedrichs discretization as in \eqref{eq:rhoprime}--\eqref{eq:Jdis}
of the equation
	$$
\pa_t \widetilde{\rho} + \pa_x \widetilde{J} = 0, \qquad
\widetilde{J} = (-\pa_x A(\pa_x\widetilde{S})+a(\pa_x\widetilde{S})\widetilde{S}).
	$$
\end{theorem}
We notice that the limit $\delta v\to 0$ is mandatory to recover the 
similar scheme as in the previous Section. This is due to the approximation
error of the trapezoidal rule.

\begin{proof} 
We actually prove that for $\delta v$ fixed, the
limit $\eps\to 0$ of the kinetic scheme furnishes a discretized 
version of the scalar conservation law as in the previous Section 
\eqref{eq:rhoprime}--\eqref{eq:Jdis} but with a macroscopic velocity
which differs from \eqref{eq:ai12} up to a $O(\delta v^2)$ term.
The Theorem is then proved by letting $\delta v$ going to $0$.
The proof is divided into several steps.

{\it (i) Nonnegativity.} Assume $f_{ij}^{n}\ge 0$ for 
all $i,j$. Since the function $E$ is nonnegative, we deduce that 
$x\mapsto \calE(v,x)$ is non-decreasing for all $v$ therefore
with \eqref{eq:ui12}, we deduce that $E_{ij}^{n} \geq 0$.
As a direct consequence of \eqref{PiErho}, we have $\Pi_{ij}^{n} \geq 0$.
Then from \eqref{fijstep2}, we conclude that $f_{ij}^{n+1/2}$ is nonnegative.
Next, from \eqref{fijstep1}, we deduce that provided 
$\lambda V_M \leq 1$, $f_{ij}^{n+1}$ is a convex
combination of $f_{ij}^{n+1/2}$, $f_{i-1,j}^{n+1/2}$ and $f_{i+1,j}^{n+1/2}$.
Then $f_{ij}^{n+1} \geq 0$.

{\it (ii) Mass conservation.} 
We recall that thanks to our choice of $e_{ij}^n$ in \eqref{defeij}, 
we have from \eqref{fijstep2} that $\rho_i^{n+1/2} = \rho_i^{n}$.
Summing \eqref{fijstep1} over $i=0,\ldots,N_x$ and $j=0,\ldots,N_v$, 
we deduce (with boundary conditions $f_{0j}=f_{N_xj}=0$) that
	$$
\delta x \sum_{i=0}^{N_x} \rho_i^{n+1} = \delta x \sum_{i=0}^{N_x} \rho_i^{n+1/2}.
	$$
Then the scheme is conservative:
	$$
\delta x \sum_{i=0}^{N_x} \rho_i^{n+1} = \delta x \sum_{i=0}^{N_x} \rho_i^{n} =M.
	$$

{\it (iii) Estimates.}
Since $\rho_i^{n+1/2}=\rho_i^{n}$, we have $S_i^{n+1/2}=S_i^{n}$ and $\nu_i^{n+1/2}=\nu_i^{n}$.
Due to the mass conservation, we still have the bound in \eqref{eq:boundnu}, i.e. for all $i$ and $n$,
\beq\label{eq:boundnu2}
\delta x \sum_{j=1}^i |\nu_i^n| \leq M w_0.
\eeq

Applying the operator $I_\delta$ on equation \eqref{fijstep1}, we deduce, 
\beq\label{eq:rho1}
\rho_i^{n+1}=\rho_i^{n+1/2} + \frac{\lambda V_M}{2} (\rho_{i+1}^{n+1/2}
-2\rho_i^{n+1/2}+\rho_{i-1}^{n+1/2})
- \frac{\lambda}{2} (I_\delta ((v_jf_{i+1,j}^{n+1/2})_j)-I_\delta ((v_jf_{i-1,j}^{n+1/2})_j)).
\eeq
We introduce the quantity 
	$$
\gamma_{i+1/2}^{n+1/2} = \begin{cases}\dfrac{I_\delta ((v_j(f_{i+1,j}^{n+1/2}+f_{i,j}^{n+1/2}))_j) }{\rho_{i+1}^{n+1/2}+\rho_i^{n+1/2}},
	& \mbox{if } \rho_{i+1}^{n+1/2}+\rho_i^{n+1/2} \neq 0 \\
		0 & \mbox{otherwise} \end{cases}
	$$
We clearly have that $0\leq \gamma_{i+1/2}^{n+1/2} \leq V_M$ and by linearity
$$
I_\delta ((v_j(f_{i+1,j}^{n+1/2}+f_{i,j}^{n+1/2}))_j) 
= \gamma_{i+1/2}^{n+1/2}(\rho_{i+1}^{n+1/2}+\rho_i^{n+1/2})
= \gamma_{i+1/2}^{n+1/2}(\rho_{i+1}^{n}+\rho_i^{n}).
$$
Finally, the equation satisfied by the sequence $(\rho_i^n)_{i,n}$ in \eqref{eq:rho1}
rewrites
\beq\label{eq:rhostep1}
\rho_i^{n+1}=\rho_i^{n} + \frac{\lambda V_M}{2} (\rho_{i+1}^{n}-2\rho_i^{n}+
\rho_{i-1}^{n}) - \frac{\lambda}{2} \big(\gamma_{i+1/2}^{n+1/2} (\rho_{i+1}^{n}+\rho_i^{n})
- \gamma_{i-1/2}^{n+1/2} (\rho_i^{n}+\rho_{i-1}^{n})\big).
\eeq
As in the proof of Theorem \ref{th:convmacro}, we introduce the quantity
$M_i^n=\delta x \sum_{k=0}^i \rho_k^n$.
By definition of $M_i^n$, we have $\delta x (\rho_{i+1}^n+\rho_i^n)=M_{i+1}^n-M_{i-1}^n$.
Therefore, summing \eqref{eq:rhostep1} with vanishing boundary conditions, we deduce
\beq\label{eq:Mcin}
M_i^{n+1}=M_i^{n}(1-\lambda V_M) + \frac{\lambda}{2} M_{i+1}^{n} (V_M-\gamma_{i+1/2}^{n+1/2}) 
+ \frac{\lambda}{2} M_{i-1}^{n}(V_M+\gamma_{i+1/2}^{n+1/2}).
\eeq
Thus $M_i^{n+1}$ is a convex combination of $M_{i-1}^{n}$, $M_i^{n}$ and 
$M_{i+1}^{n}$.
It is then obvious by an induction on $n$ to deduce that for all $i$ and $n$, 
$0\leq M_i^n\leq M$ and that we have a $BV$-estimate on the sequence $(M_i^n)$.

From \eqref{eq:Vwdis2}, we have
	\beq\label{eq:dxSnu}
-\frac{\pa_xS_{i+1/2}^n-\pa_xS_{i-1/2}^n}{\delta x} + \nu_i^n = \rho_i^n.
	\eeq
Summing this latter equation on $i$, we deduce, using $S_1^n=S_0^n$,
	\beq\label{eq:dxSM}
\pa_xS_{i+1/2}^n = \sum_{k=1}^i \delta x \nu_k^n - M_i^n.
	\eeq
Thus, we have the bound for all $i$, $n$
	\beq\label{bounddxS}
|\pa_xS_{i+1/2}^n|\leq M(1+w_0).
	\eeq
We deduce then that there exists a nonnegative constant $C$ independant 
on $\eps$ such that
	$$
0\leq E_{ij}^n \leq C.
	$$
From \eqref{hypM1} and assuming that $E\in C^2(\RR^2)$, we deduce 
from error estimates for the trapezoidal rule \eqref{errortrap} that
	$$
I_\delta((E_{ij}^{n})_j)= 1+O(\delta v^2),
	$$
where the constant in $O(\delta v^2)$ depends actually on
$\|\pa_xS\|_\infty$. However, from \eqref{bounddxS}, this bound depends 
only on $M$, $w_0$ and $V_M$.

{\it (iv) Passing to the limit $\eps\to 0$.}
We will denote with a tilde $\widetilde{\ }$ all the limits of considered quantities
when they exist.

From the $L^\infty\cap BV$ bound independant of $\eps$ on the sequence $(M_i^n)_{i,n}$,
we deduce that we can extract a subsequence that 
converges strongly in $L^1_{loc}$ as $\eps\to 0$ to $(\widetilde{M}_i^n)_{i,n}$.
Moreover, since the sequence $(\gamma_{i+1/2}^{n+1/2})_{i,n}$ is bounded in $L^\infty$
independantly of $\eps$, we can extract a subsequence converging in 
$L^\infty-weak *$ as $\eps\to 0$ to $(\widetilde{\gamma}_{i+1/2}^{n+1/2})_{i,n}$.
Taking the limit $\eps\to 0$ in \eqref{eq:Mcin}, we deduce 
that the limit sequences $(\widetilde{M}_i^n)_{i,n}$ and
$(\widetilde{\gamma}_{i+1/2}^{n+1/2})_{i,n}$ satisfy the same relation \eqref{eq:Mcin}.
Defining $\widetilde{\rho}_i^n=(\widetilde{M}_i^n-\widetilde{M}_{i-1}^n)/\delta x$, 
this sequence satisfies equation \eqref{eq:rhostep1}. Moreover, we have 
the weak convergence in $\smes$ of $(\rho_i^n)_{i,n}$ towards 
$(\widetilde{\rho}_i^n)$ as $\eps\to 0$.

Using \eqref{eq:nu}, we deduce, by using $\rho_k^n=(M_k^n-M_{k-1}^n)/\delta x$,
	$$
\nu_i^n = \sum_{k=1}^{N_x} M_k^n
\int_{(i-1-k)\delta x}^{(i-k)\delta x} \frac{w(z+\delta x)-w(z)}{\delta x}\,dz.
	$$
Since the function $w$ is bounded, we deduce that the integral in the
right hand side is bounded.
Therefore, from the bound on $(M_i^n)_{i,n}$, we deduce a $BV$-bound on the sequence 
$(\nu_i^n)_{i,n}$, whose we can extract a subsequence that converges strongly in 
$L^1_{loc}$ as $\eps\to 0$.
We deduce using moreover \eqref{eq:dxSM} that, up to a 
subsequence, $(\pa_xS_{i+1/2}^n)_{i,n}$ converges strongly in $L^1_{loc}$ as $\eps\to 0$.
Obviously, the limit satisfies relation \eqref{eq:dxSnu}.
Then, we can pass to the limit in \eqref{eq:ui12} and in \eqref{defeij} to 
deduce the $L^\infty-weak*$ convergence of $(E_{ij}^{n})_{i,j,n}$ and of
$(e_{ij}^{n})_{i,j,n}$.
By the same token, we have the $L^\infty-weak*$ convergence of 
$(\Pi_{ij}^{n})_{i,j,n}$ and $(\widehat{a_\delta}_i^{n})_{i,n}$ towards limits
still satisfying \eqref{Pidev} and \eqref{adelta} with a tilde on all
quantities.
Then we can rewrite the limiting expression
	$$
\widetilde{\Pi}_{ij}^{n} = \widetilde{e}_{ij}^{n} \widetilde{\rho}_i^{n}.
	$$

From \eqref{fijstep2}, we have
	$$
I_\delta((v_jf_{ij}^{n+1/2})_j) = e^{-\delta t/\eps} I_\delta((v_jf_{ij}^{n})_j)
+ (1-e^{-\delta t/\eps}) I_\delta((v_j\Pi_{ij}^{n})_j).
	$$
Since $I_\delta((v_jf_{ij}^{n})_j)\leq V_M M$, we can take the weak limit 
as $\eps\to 0$ and deduce that
	$$
I_\delta((v_jf_{ij}^{n+1/2})_j) \rightharpoonup I_\delta((v_j\widetilde{\Pi}_{ij}^{n})_j)
= \widetilde{a_\delta}_{ij}^{n}  \widetilde{\rho}_i^{n}.
	$$
Finally, when $\eps\to 0$, we have using \eqref{eq:rho1} that
	$$
\widetilde{\rho}_i^{n+1} = \widetilde{\rho}_i^{n} + 
\frac{\lambda V_M}{2} ( \widetilde{\rho}_{i+1}^{n} -2 \widetilde{\rho}_i^{n}
+ \widetilde{\rho}_{i-1}^{n}) - \frac{\lambda}{2}
(\widetilde{a_\delta}_{i+1}^{n} \widetilde{\rho}_{i+1}^{n}
- \widetilde{a_\delta}_{i-1}^{n} \widetilde{\rho}_{i-1}^{n}).
	$$
Moreover, due to the $L^\infty$-bound on $(\pa_xS_i^n)_{i,n}$ \eqref{bounddxS} and
the error estimates on the trapezoidal rule \eqref{errortrap}, we deduce that
	$$
\widetilde{a_\delta}_i^{n} = \frac{A(\widetilde{\pa_xS}_{i+1/2}^{n})-A(\widetilde{\pa_xS}_{i-1/2}^{n})}
{\widetilde{\pa_xS}_{i+1/2}^{n}-\widetilde{\pa_xS}_{i-1/2}^{n}} \big( 1+ O(\delta v^2) \big).
	$$
Thus when $\delta v\to 0$, the limit $\widetilde{\rho}_i^n$
satisfies a Lax-Friedrichs discretization of the problem 
$\pa_t\rho + \pa_x(\widehat{a}\rho)=0$, where $\widehat{a}$ is discretized by \eqref{eq:ai12}.
\end{proof}

\Section{Numerical simulations}\label{simul}
We present in this Section some numerical examples to illustrate our 
results. In particular, we present two examples with applications
in biology or plasma physics, where $w=0$ or $w=W$.

\subsection{Simulation of an aggregation equation}
In this subsection we consider the case $W=-\frac 12 |x|$ and $a=\mbox{id}$. 
Then the equation writes
	$$
\pa_t\rho + \pa_x((W'*\rho) \rho) = 0.
	$$
This equation appears in several applications in biology or physics, see for
instance \cite{NPS} where this system is the high field limit of a
Vlasov-Poisson-Fokker-Planck system, the quantity $S=W*\rho$ being the solution
of the Poisson equation.
In biology, it can be seen as a Patlack-Keller-Segel model without diffusion, the 
quantity $S$ being the chemoattractant concentration.

Numerical scheme \eqref{eq:rhoprime}--\eqref{eq:Jdis} is implemented.
We notice that in the case $a(x)=x$, \eqref{eq:ai12} rewrites
$a_{i+1/2}^n = \frac 12 \big( \pa_xS_{i+1}^n +\pa_xS_{i}^n \big)$.
Numerical results are display in Figure \ref{figvpfp0} for 
two different initial data. In Figure \ref{figvpfp0} left, we take
$\rho^{ini}(x)=e^{-10x^2}$. We observe that the initial bump stiffens and 
the solutions blows up in finite time to form one stationary single Dirac.
In Figure \ref{figvpfp0} right we take 
$\rho^{ini}(x)=e^{-10(x-1.25)^2}+0.8 e^{-20x^2}+e^{-10(x+1)^2}$.
As for the previous initial data, the initial bumps blow up and
collapse in one single Dirac mass in finite time.

\begin{figure}[htp]
\begin{center}
  \includegraphics[width=2.5in]{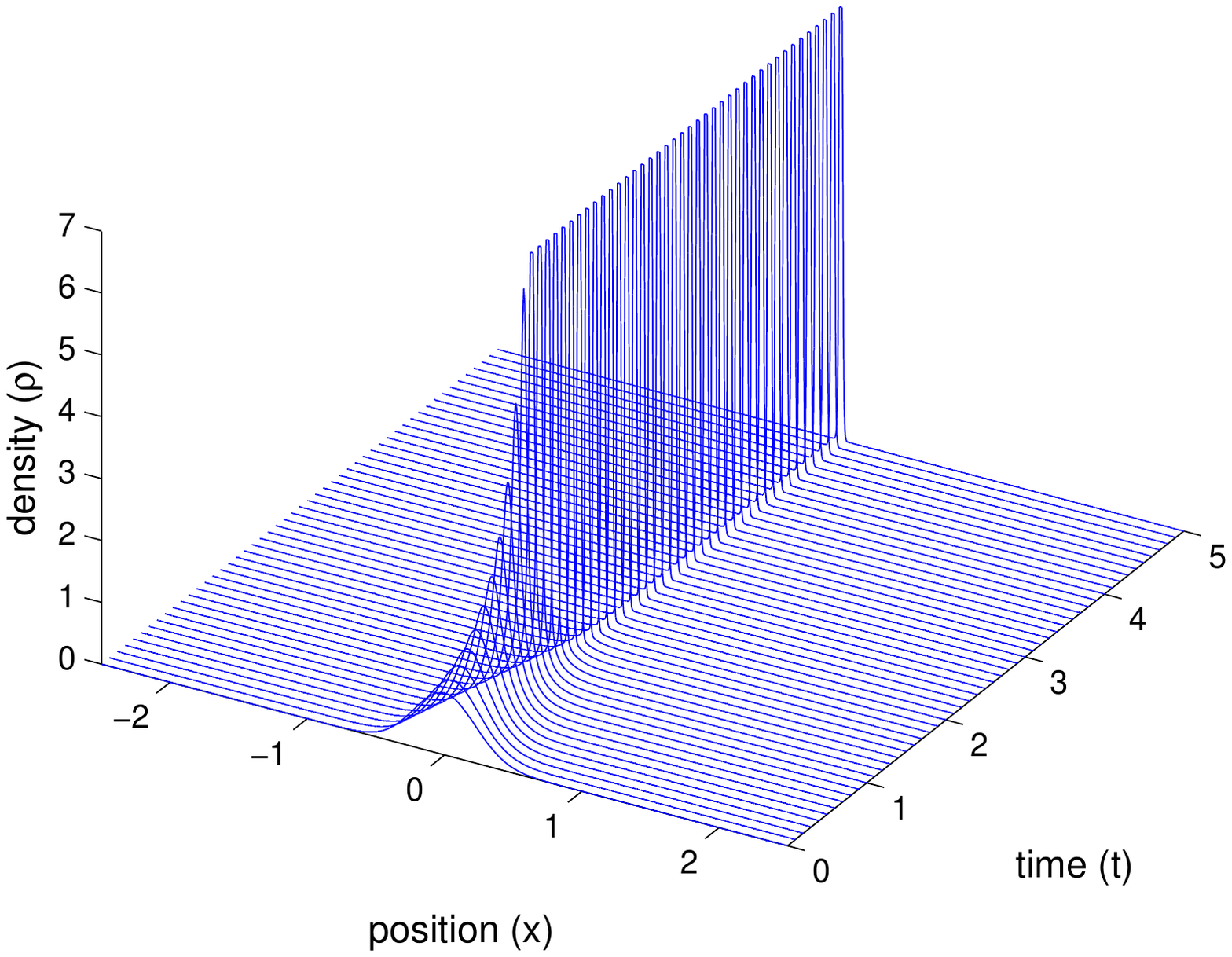}
  \includegraphics[width=2.5in]{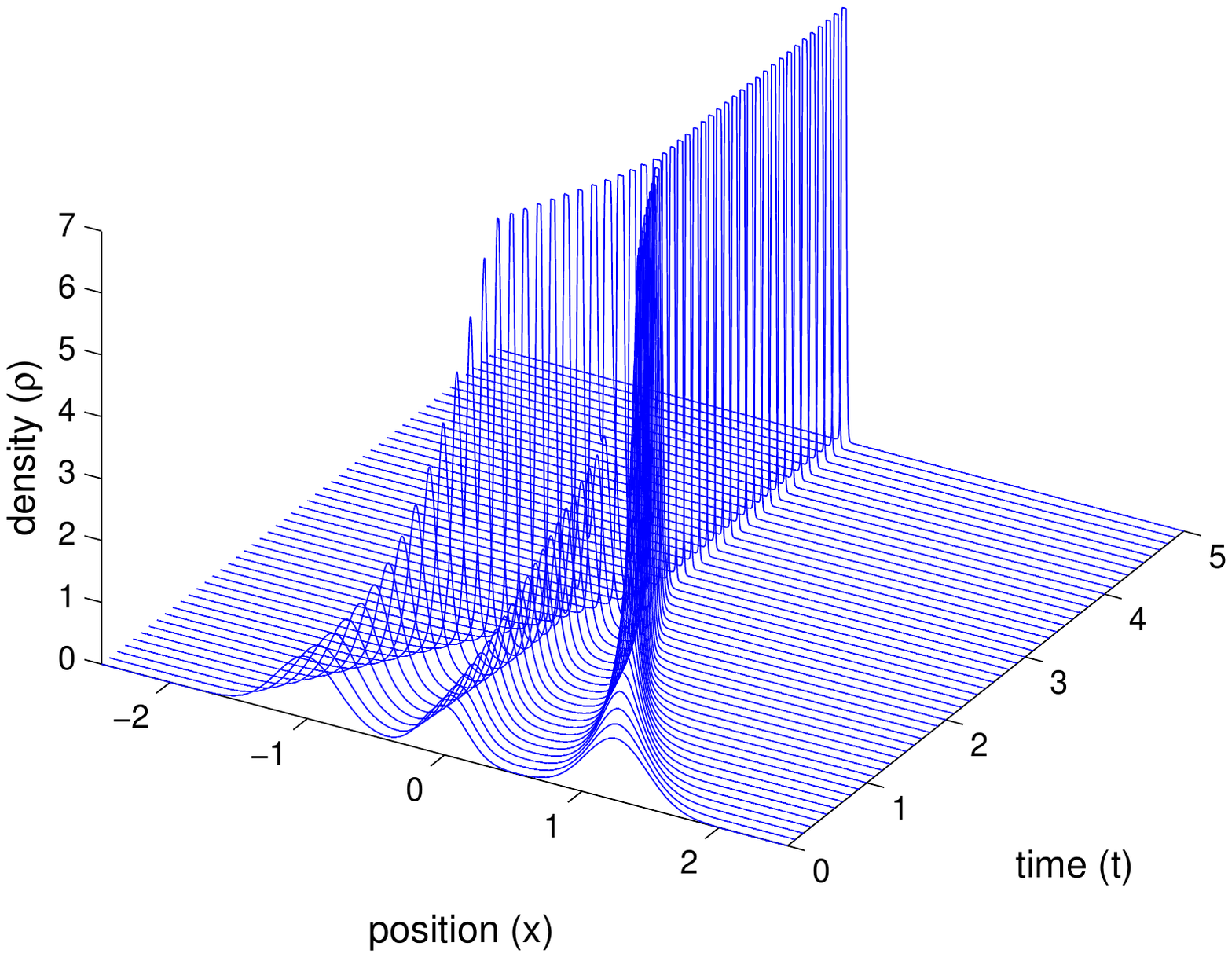} \\
  \caption{Dynamics of the density $\rho$ with two different intial data
  in the case $W=-\frac 12 |x|$}\label{figvpfp0}
\end{center}
\end{figure}

In this case $W''=-\delta_0$. Then, setting $S=W*\rho$, we have 
that $-\pa_{xx}S=\rho$ and $-\pa_xS$ is an antiderivative of $\rho$.
Then, integrating the aggregation equation, we can rewrite it as
$$
\pa_t \pa_xS + \frac 12 \pa_x \big( \pa_x S\big)^2 = 0.
$$
We recognize the Burgers equation for $\pa_xS$.
Moreover, in this particular case where $\nu_i^n=0$, we can 
deduce from \eqref{eq:rhoprime}--\eqref{eq:Vwdis} a scheme on
$(\pa_xS_i^n)_{i,n}$. First, \eqref{eq:Jdis} rewrites
$$
J_{i+1/2}^n = - \frac{1}{\delta x} \big(A(\pa_xS_{i+1}^n)
-A(\pa_xS_i^n)\big).
$$
Moreover, denoting 
$\rho_{i+1/2}^n= \frac 12 (\rho_i^n+\rho_{i+1}^n)$, we deduce from 
\eqref{eq:Vwdis} and \eqref{paW}
$$
\rho_{i+1/2}^n = -\frac {1}{\delta x} \big(\pa_xS_{i+1}^n-\pa_xS_i^n\big).
$$
We deduce
$$
\sum_{k=0}^i \rho_{k+1/2}^n = - \frac {1}{\delta x} 
\big(\pa_xS_{i+1}^n - \pa_xS_0^n\big).
$$
Equation \eqref{eq:rhoprime} implies straightforwardly
$$
\rho_{i+1/2}^{n+1} = \rho_{i+1/2}^n(1-\lambda c)+ \frac{\lambda}{2}c
(\rho_{i-1/2}^n+\rho_{i+3/2}^n) + \frac{\lambda}{4} (J_{i-1/2}^n-J_{i+3/2}^n).
$$
Summing this latter equation and using the boundary conditions 
$\pa_xS_0^n=0$, we deduce
$$
\pa_xS_i^{n+1} = \pa_xS_{i}^n(1-\lambda c)+ \frac{\lambda}{2}c
(\pa_xS_{i-1}^n+\pa_xS_{i+1}^n) - 
\frac{\lambda}{4} \big(A(\pa_xS_{i+1}^n)-A(\pa_xS_{i-1}^n)\big).
$$
In the case at hand where $a=id$, we have $A(x)=x^2/2$, 
we recognize the well-known Lax-Friedrichs 
discretization for the Burgers equation. 
Here we have $c=M$ where $M$ is the total mass of the system.
Then the numerical results of Figure \ref{figvpfp0} was expected;
we recover the convergence in finite time in a single Dirac mass 
as established for instance in \cite[Section 4]{Carrillo}.


\subsection{A kinetic model for chemotaxis}\label{modelchemo}
Let us consider the so-called Othmer-Dunbar-Alt model, 
describing the motion of cells by chemotaxis, in one dimension. 
This model has been used since the 80's when it has been observed
that the motion of bacteria is due to the alternance of straigth swim
in a given direction, called {\it run phase}, with cells reorientation
to choose a new direction, called {\it tumble phase}.
This system governs the dynamics of the distribution function $f_\eps$.
In the hyperbolic scaling, it writes:
	$$
\pa_t f_\eps + v\pa_xf_\eps = \frac 1\eps \int_V \big(T[v'\to v] f_\eps(v') 
-T[v\to v'] f_\eps(v)\big) \,dv'.
	$$
In this equation $T[v'\to v]$ is the turning rate, corresponding to the
probability of cells to change their velocities from $v'$ to $v$
during a tumble phase. In this work, we consider the model proposed 
by Dolak and Schmeiser \cite{dolschmeis} where
	$$
T[v'\to v]= \phi(v'\pa_xS_\eps).
	$$
The function $\phi$ is given and assumed to be in $C^1(\RR)$.
In this equation, the quantity $S_\eps$ corresponds to the chemoattractant 
concentration which solves the elliptic equation
	\begin{equation}\label{eq:S}
-\pa_{xx}S_\eps+S_\eps=\rho_\eps,
	\end{equation}
where $\rho_\eps=\int_V f_\eps(v)\,dv$ is the density of cells. This latter equation 
can be rewritten $S_\eps=W*\rho_\eps$ for $W=\frac 12 e^{-|x|}$;
therefore we have $W=w$ in \eqref{hyp}.

The velocity $v$ is assumed to have a constant modulus, therefore
the set of velocities in one dimension is given by $V=\{-v,v\}$.
Then the kinetic equation rewrites in one dimension:
	\begin{equation}\label{chemokinetic}
\pa_t f_\eps + v \pa_x f_\eps = \frac 1\eps (\phi(-v\pa_xS_\eps)f_\eps(-v) 
-\phi(v\pa_xS_\eps)f_\eps(v)).
	\end{equation}
Then the density of cells is defined by $\rho_\eps:=f_\eps(v)+f_\eps(-v)$.
Up to a rescaling, we assume that the turning rate satisfies
	\begin{equation}\label{assumpphi}
\phi(x)= \frac 12 + \phi_1(x), \mbox{ with } \phi_1(-x)=-\phi_1(x)
	\end{equation}
so that $\phi(x)+\phi(-x)=1$, for all $x\in \RR$. 
Then, we can rewrite \eqref{chemokinetic} as
	$$
\pa_t f_\eps + v \pa_x f_\eps = \frac 1\eps ( \phi(-v\pa_xS_\eps)\rho_\eps
- f_\eps(v)).
	$$
We recognize the form \eqref{eq:kin} for the kinetic equation with $E(v,x)=\phi(-v x)$. 
Then the limiting model when $\eps\to 0$ is given by (see Theorem \ref{ConvKinet})
	\beq\label{chemomacro}
\pa_t \rho + \pa_x J = 0, \qquad 
J = -\pa_x A(\pa_xS) + a(\pa_xS) S, \qquad a(x) = -v\phi_1(vx).
	\eeq
This equation have been studied in \cite{jamesnv}.

\begin{remark}
As in the first example of this Section, we can recover an equation for the potential $S=W*\rho$, which turns out here
to be nonlocal. Indeed, taking the convolution with $W=\frac 12 e^{-|x|}$ of \eqref{chemomacro}, we obtain
	$$
\pa_t S + A(\pa_xS) - W*A(\pa_xS) + \pa_x W*(a(\pa_xS)S) = 0.
	$$
Then by recombining \eqref{eq:S} and \eqref{chemomacro}, this latter
equation can rewrite
	$$
\pa_t S - \pa_{txx}S + \pa_x\big[-\pa_xA(\pa_xS)+a(\pa_xS)S\big] = 0.
	$$
It bears some resemblance with the well-known Camassa-Holm equation \cite{camassa}, and exhibits the same 
peakon-like solutions. However the underlying dynamics is completely different and in the present case peakons  
collapse. Notice also that there are no anti-peakons because of the positivity of $\rho$.
\end{remark}

\subsubsection{Attractive case}
The computational domain is assumed to be $[-2.5,2.5]$ and the velocity
$v$ is normalized to 1. 
We consider the function $a(x)=-\phi_1(x)=2/\pi \mbox{ Arctan}(10 x)$, which 
clearly satisfies (\ref{hyp_a}).

We first consider the macroscopic model \eqref{chemomacro} and discretize
the system thanks to \eqref{eq:rhoprime}--\eqref{eq:Vwdis}, with $\nu_i^n$
replaced by $S_i^n$ in \eqref{eq:Vwdis}. Figure \ref{fig_macro} displays
the numerical results for the following initial data:
$\rho^{ini}(x)=e^{-10(x-0.7)^2}+e^{-10 (x+0.7)^2}$ (left) and 
$\rho^{ini}(x)=e^{-10(x-1.25)^2}+0.8 e^{-20 x^2}+e^{-10(x+1)^2}$ (right).
As expected, we have a fast blow up of regular solution and a finite-time
collapse in a single Dirac mass.

\begin{figure}[htp]
\begin{center}
  \includegraphics[width=2.5in]{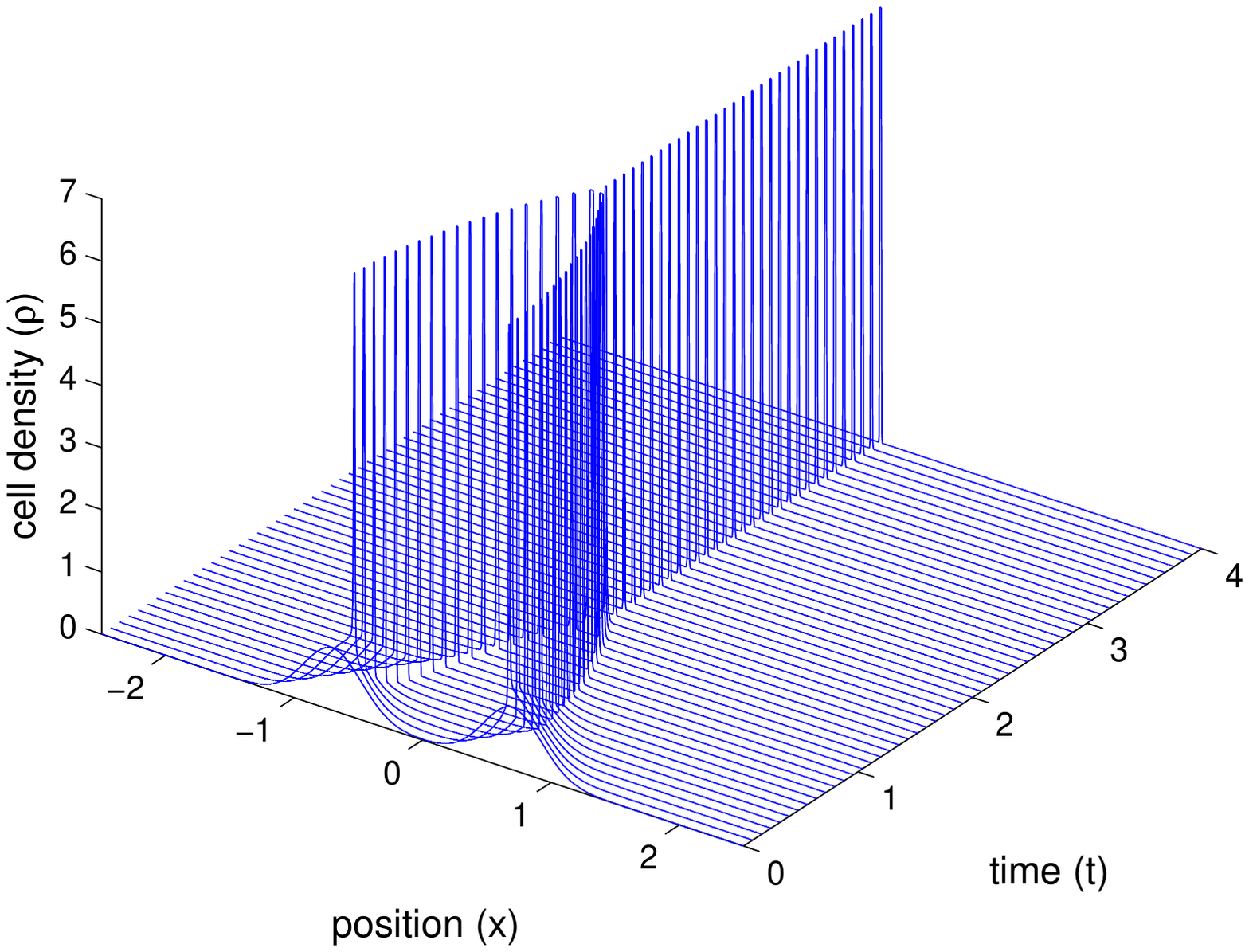}
  \includegraphics[width=2.5in]{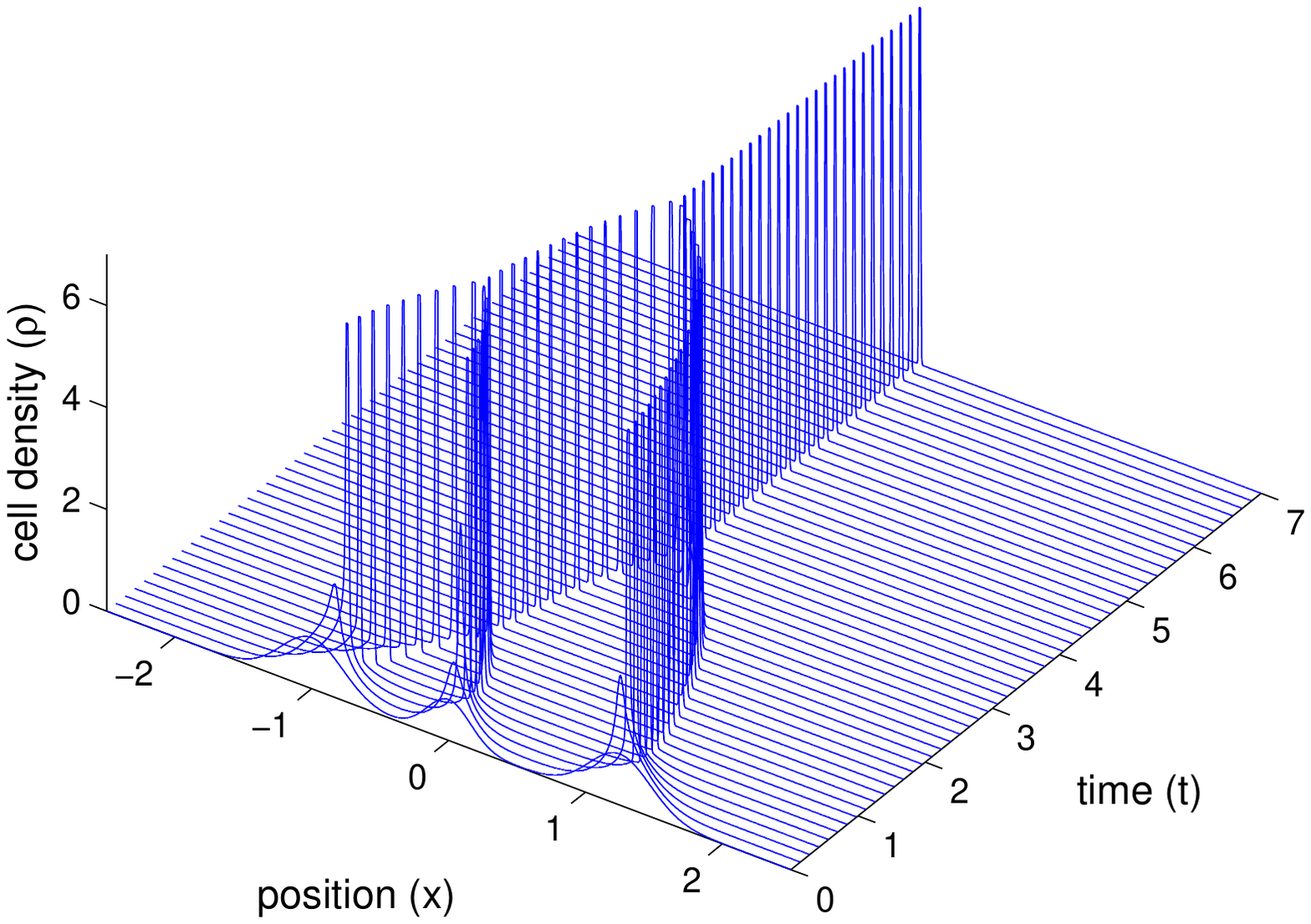}\\
  \caption{Macroscopic model \eqref{chemomacro}: cell density $\rho$ for two 
    regular bumps (left) and three regular bumps (right) initial data.}\label{fig_macro}
  \end{center}
\end{figure}

The behaviour of such Dirac solutions can be recovered by studying 
solutions in the form $\rho(t,x)=\sum_{i=1}^n m_i \delta(x-x_i(t))$. Then we have
$S(t,x)=W*\rho(t,x) = \frac 12 \sum_{i=1}^n m_i e^{-|x-x_i(t)|}$. After
straightforward computations, we deduce from the expression in
\eqref{chemomacro} that 
	$$
J=-\sum_{i=1}^n [A(\pa_xS)]_{x_i} \delta(x-x_i(t)),
	$$
where the notation $[f]_{x_i}$ denotes the jump of the function $f$ at
the point $x_i$. In particular, we have that $\rho$ satisfies 
system \eqref{chemomacro} provided,
	$$
m_i x'_i(t)=- [A(\pa_xS)]_{x_i}.
	$$
Moreover, the function 
$a$ being increasing and odd, the function $A$ is strictly convex and 
can be chosen even. Then, equilibrium states satisfy
	$$
-[A(\pa_xS)]_{x_i} = A\Big(\frac 12 \big(m_i+\sum_{j\neq i} m_j e^{-|x_j-x_i|}
\big) \Big) - A\Big(\frac 12 \big(-m_i+\sum_{j\neq i} m_j e^{-|x_j-x_i|}
\big) \Big)=0.
	$$
This equality is true only if $\sum_{j\neq i} m_j e^{-|x_j-x_i|}=0$, which implies $n=1$. 
Therefore stationary states are given by a single stationary Dirac mass.
Convergence towards this equilibrium is proved in \cite{jamesnv,Carrillo}.

Then we consider the kinetic framework and implement the scheme described in 
Section \ref{APsec}.
In Figure \ref{fig2} and \ref{fig3} we display the dynamics of the cell density
$\rho$ for the regular initial data with two bumps:
$\rho^{ini}(x)=e^{-10(x-0.7)^2}+e^{-10 (x+0.7)^2}$
and an initial distribution function given by $f^{ini}(x,v)=\frac 12 \rho^{ini}(x)$.
We plot in the left part of the figures the numerical results corresponding
to the macroscopic model \eqref{chemomacro}, whereas the right part corresponds
to numerical solution of \eqref{chemokinetic}.
For the macroscopic case (left), we notice that the blow up occurs fastly. 
After a small time, solution is composed of 2 peaks 
which can be considered as numerical Dirac masses.
Then the Dirac masses move and collapse in finite time.
In the kinetic case (right), the solution does not blow up, as it is expected.
However, the behaviour is similar to the one for the macroscopic model:
we observe the formation of two aggregates that are in interaction and attract
themselves in a single aggregate in finite time. 

\begin{figure}[htp]
\begin{center}
  \includegraphics[width=2.5in]{2bumptemps.pdf}
  \includegraphics[width=2.5in]{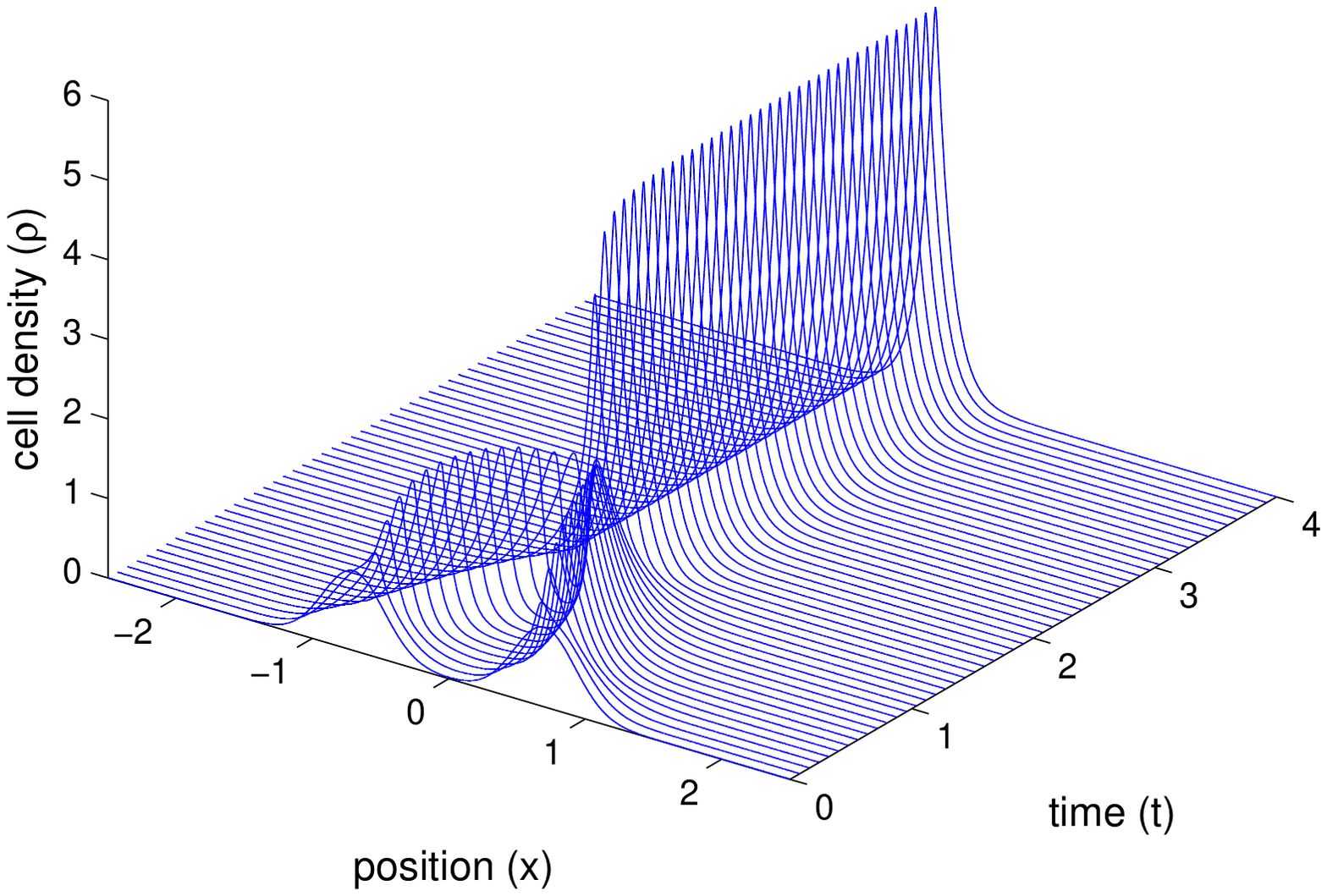}\\
  \caption{Dynamics of the cell density $\rho$ for an initial data given by a two 
    regular bumps: comparison between the macroscopic model (left) and
    the kinetic model for $\eps=0.1$ (right).}\label{fig2}
  \end{center}
\end{figure}

\begin{figure}[htp]
\begin{center}
  \includegraphics[width=2.5in]{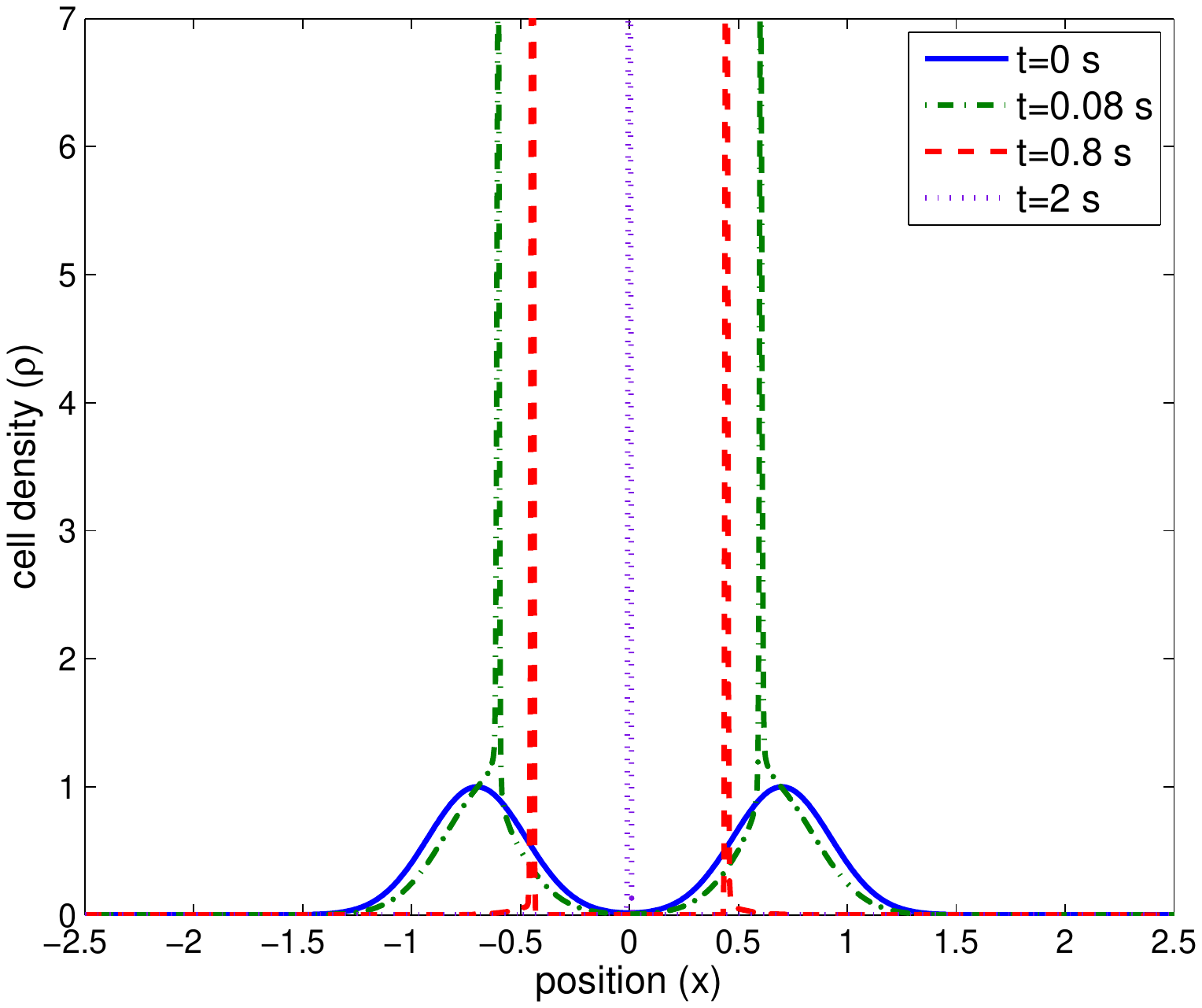}
  \includegraphics[width=2.5in]{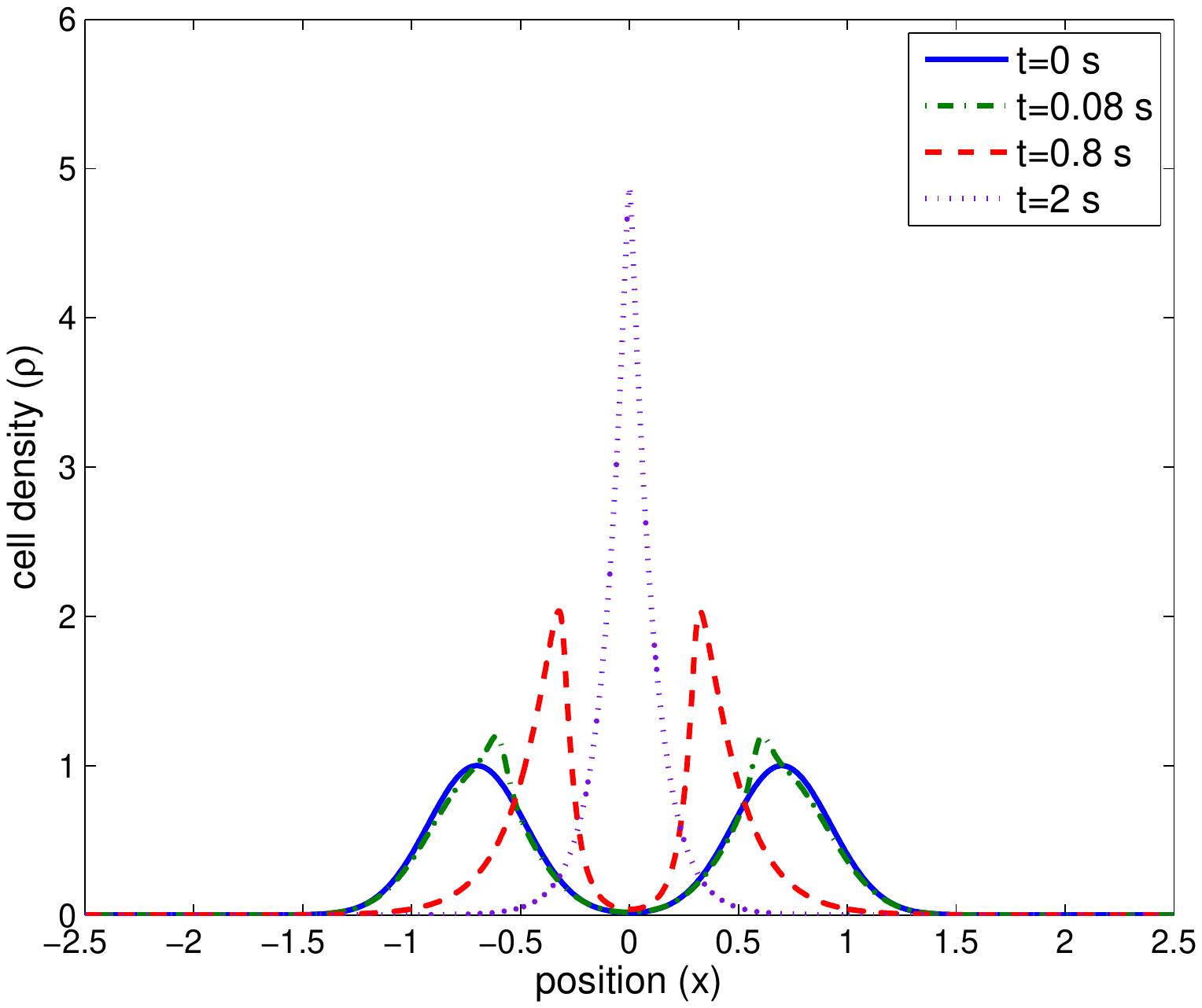}\\
  \caption{Dynamics of the cell density $\rho$ for an initial data given by a two 
    regular bumps: comparison between the macroscopic model (left) and
    the kinetic model for $\eps=0.1$ (right).}\label{fig3}
  \end{center}
\end{figure}

In Figure \ref{fig3bump} we display the numerical results for 
different values of $\eps$ and for
a regular initial data given by the three bumps:
$\rho^{ini}(x)=e^{-10(x-1.25)^2}+0.8 e^{-20 x^2}+e^{-10(x+1)^2}.$
As above, we notice the formation of 3 aggregates that merge into one 
single aggregate.
When $\eps\to 0$, we observe that the numerical solutions converges to
the one computed in the macroscopic case, which is an illustration of
the result of Theorem \ref{prop:AP}.

\begin{figure}[htp]
\begin{center}
  \includegraphics[width=2.5in]{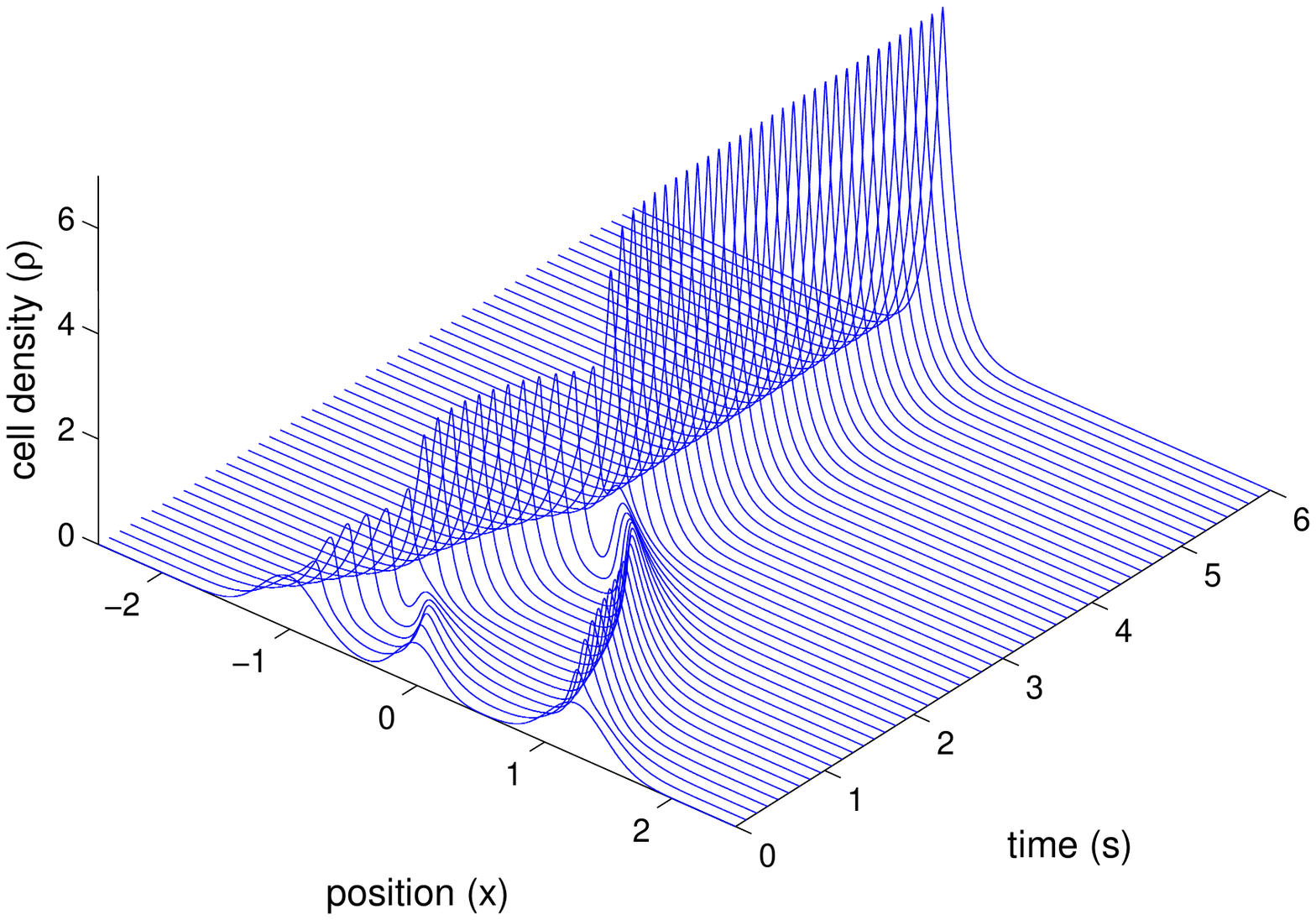}
  \includegraphics[width=2.5in]{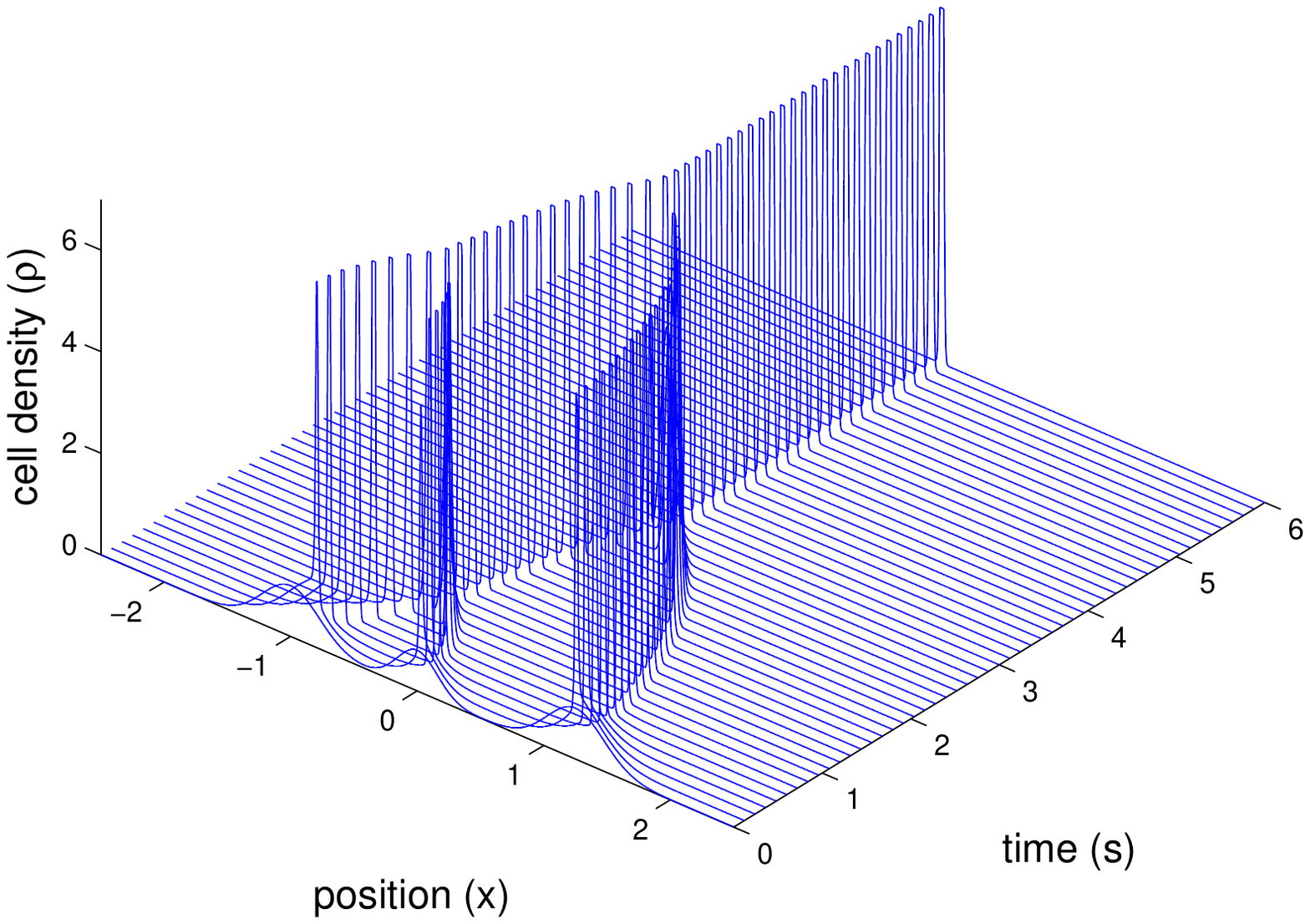}\\
  \caption{Dynamics of the cell density for an initial data given by a sum of 3
    regular bumps: simulation of the kinetic model with 
    $\eps=0.1$ (left) and $\eps=10^{-3}$ (right).}\label{fig3bump}
  \end{center}
\end{figure}

In the kinetic case, stationary state for \eqref{chemokinetic}
are given by
	$$
v\pa_xf_\eps = \frac{1}{\eps} (\phi(-v\pa_xS_\eps) \rho_\eps - f_\eps).
	$$
Summing the equation for $v$ and for $-v$, we deduce easily that
we have that $f_\eps(v)=f_\eps(-v)$. 
Then, using the expression of $\phi$ in \eqref{assumpphi}, 
the kinetic equation rewrites, for $v>0$,
	$$
v\pa_x f_\eps = -\frac{1}{\eps} \phi_1(v\pa_xS_\eps) f_\eps.
	$$
Formally, if we consider that the function $\phi_1$ is an approximation of the
function $-sign$. Then, the latter equation is a linear ODE which can be
solved easily and implies that
$f$ is given by the sum of exponential function with the tail 
$\pm\frac{1}{\eps v}$. This behaviour corresponds to what is observed 
in Figure \ref{fig3} right.

Finally, we emphasize the importance of the choice of the discretized 
macroscopic velocity. For instance, in the aggregation equation \eqref{eq:aggreg}, if instead of defining
the discretization \eqref{eq:ai12} we take $a_{i+1/2}^n=a((S_{i+1}^n-S_i^n)/\delta x)$,
we obtain Figure \ref{fig:wrong}, to be compared with Figure \ref{fig2}.
\begin{figure}[htp]
\begin{center}
  \includegraphics[width=2.5in]{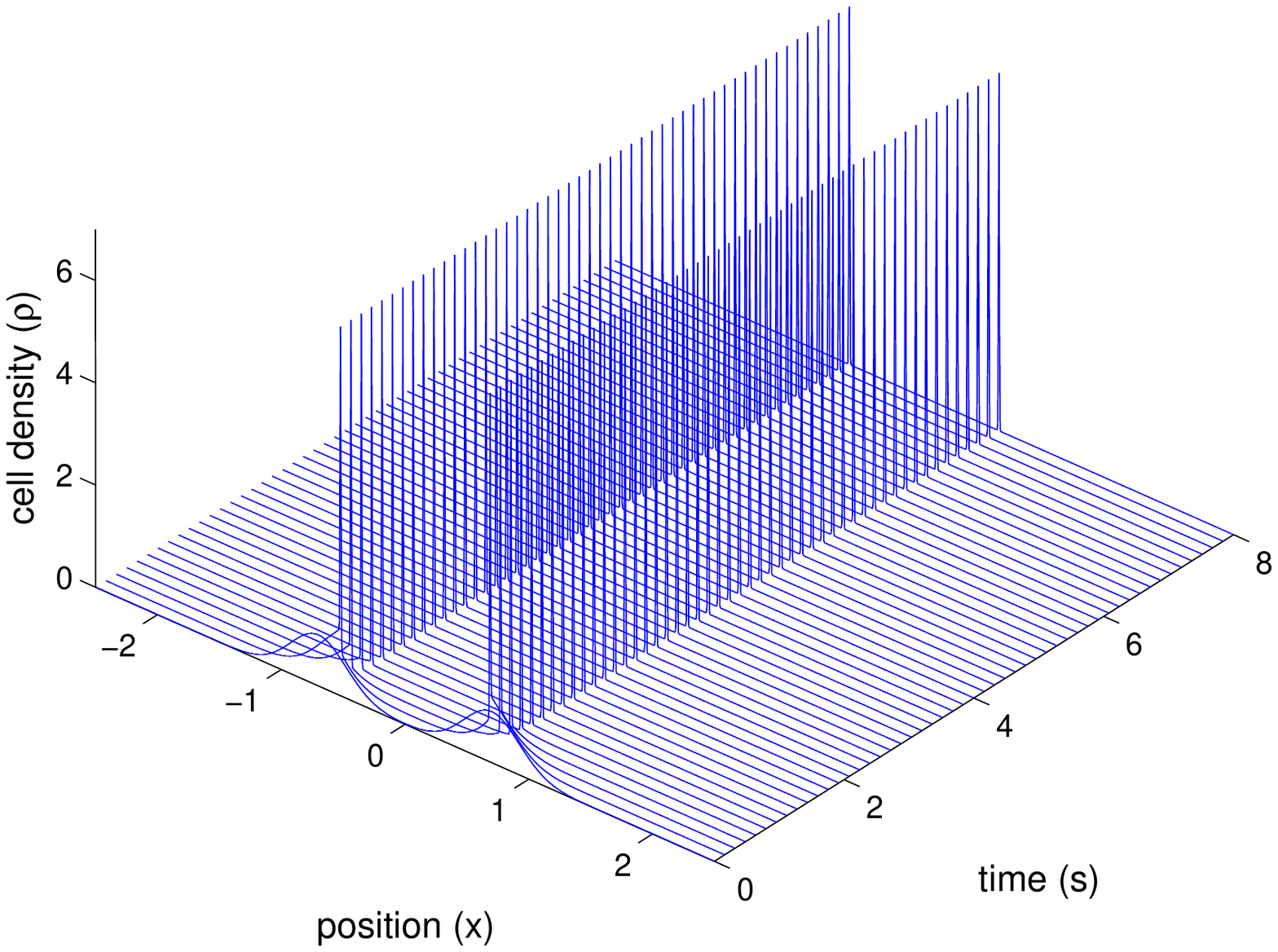}
  \caption{Wrong velocity discretization for \eqref{eq:aggreg}}
  \label{fig:wrong}
  \end{center}
\end{figure}
Concerning the kinetic model, we display in Figure \ref{fig:wrongKin} the results 
obtained when the discretization of $E_{ij}^n$ in \eqref{eq:ui12} is replaced by
$E_{ij}^n=E(v_j,\pa_xS_i^n)$, with $\pa_xS_i^n=(S_{i+1}^n-S_{i-1}^n)/2\delta x$.
We notice in Figure \ref{fig:wrongKin}(left) that for $\eps=10^{-3}$ the behaviour of the density 
remains comparable with
the macroscopic model (see Figure \ref{fig_macro}). When $\eps$ goes to zero, namely here
$\eps=10^{-5}$, we observe the same kind of result as for the macroscopic case, compare 
Figures \ref{fig:wrongKin}(right) and \ref{fig:wrong}.

\begin{figure}[htp]
\begin{center}
  \includegraphics[width=2.5in]{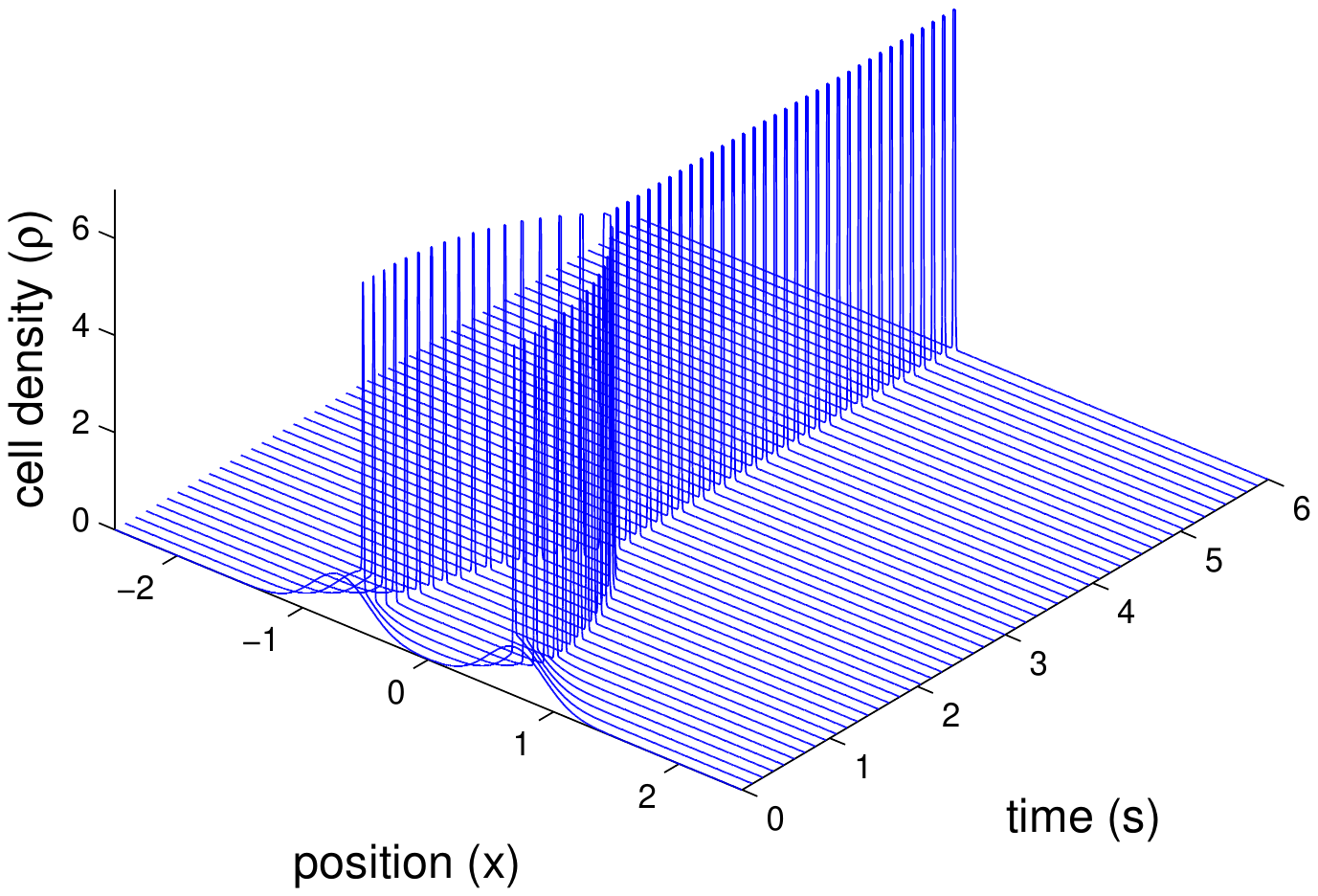}
  \includegraphics[width=2.5in]{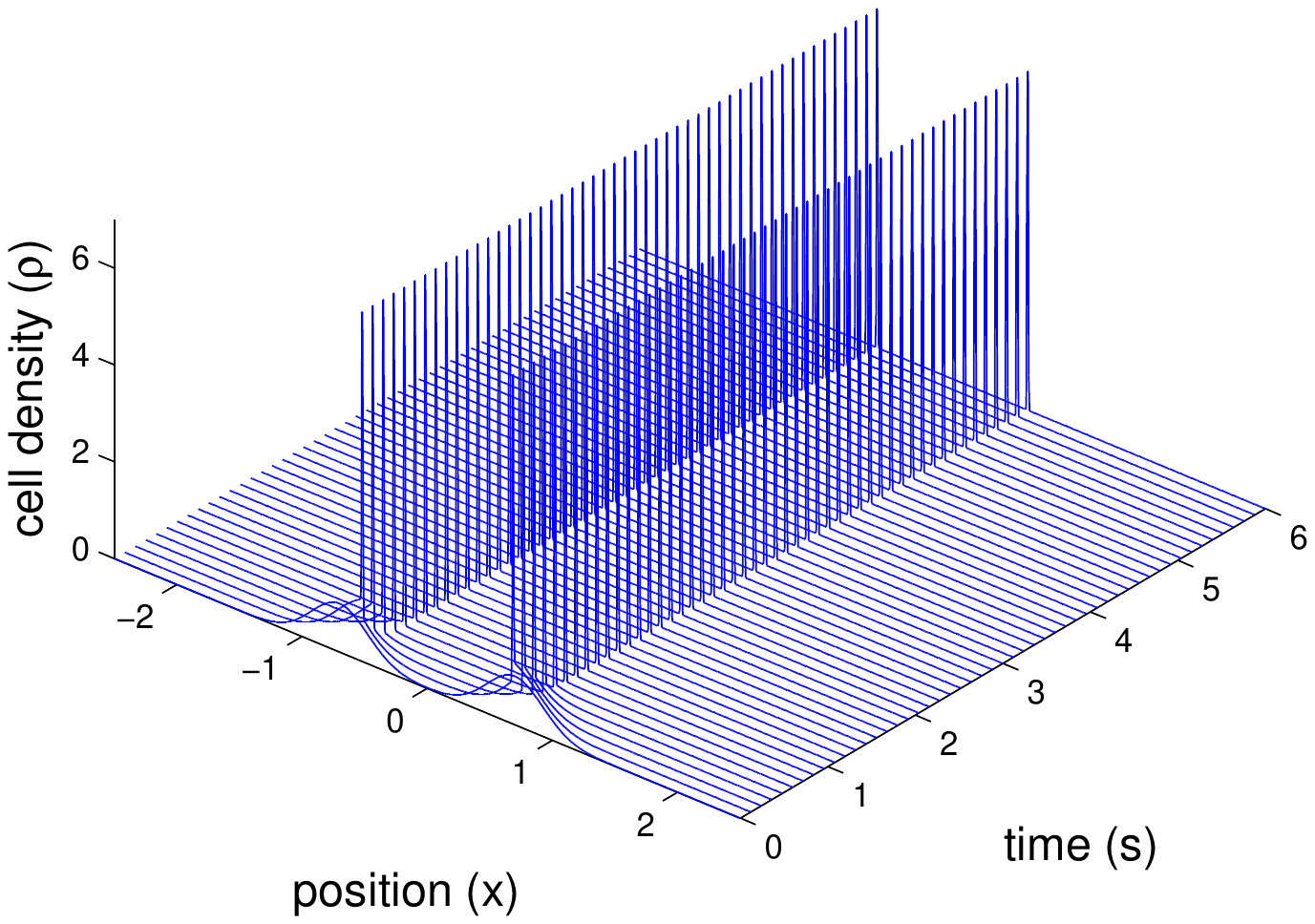}\\
  \caption{Wrong velocity discretization for kinetic model: 
	$\eps=10^{-3}$ (left) -- $\eps=10^{-5}$ (right)}\label{fig:wrongKin}
  \end{center}
\end{figure}

\subsubsection{Repulsive case}
Finally, we conclude this paper with some numerical results without any proof
in the repulsive case i.e. when the function $a$ is non-increasing.
Then assumption \eqref{hyp_a} is not satisfied and the velocity 
$x\mapsto a(\pa_xS)$ does not satisfy the one-sided Lipschitz estimate (\ref{OSLC}).
However, it can be proved (using the arguments in e.g. \cite{NPS})
that if $\rho^{ini}\in L^1\cap W^{1,\infty}(\RR)$, we have 
global in time existence of weak solutions in $L^1\cap W^{1,\infty}(\RR)$. 
Using the dynamics of Dirac masses to design a particle scheme is therefore not straghtforward in this case.
In this respect, we refer 
to \cite{bonaschi}, where gradient flow solutions are proved to be equivalent to entropy solutions of the 
Burgers equation, with a particular focus on the repulsive case.
A particle discretization is also proposed. 
It is then interesting to implement our numerical discretization 
\eqref{eq:rhoprime}--\eqref{eq:Jdis} for the macroscopic model
\eqref{chemomacro} in the repulsive case. 

Figure \ref{figrepuls1} displays the numerical results for 
$a(x)=-2/\pi \mbox{ Arctan}(10 x)$ (left) and 
for $a(x)=-2/\pi \mbox{ Arctan}(50 x)$ (right) with 
the initial data $\rho^{ini}(x)=e^{-10 x^2}$.
Figure \ref{figrepuls2} displays the dynamic of the cells density for
$a(x)=-2/\pi \mbox{ Atan}(10x)$ and for the initial data
$\rho^{ini}(x)=e^{-10(x-0.7)^2}+e^{-10(x+0.7)^2}$.

\begin{figure}[ht!]
\begin{center}
  \includegraphics[width=2.5in]{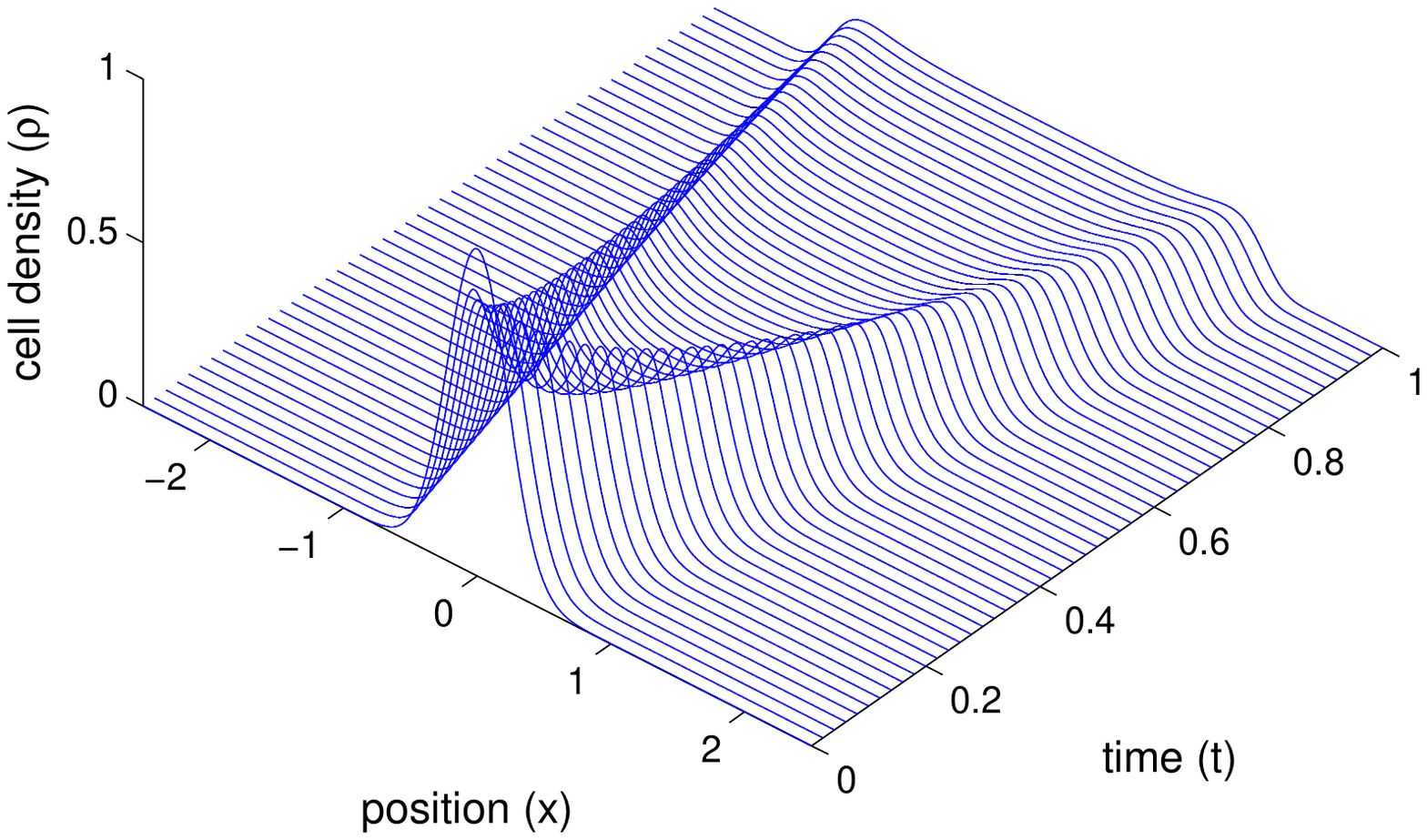}
  \includegraphics[width=2.5in]{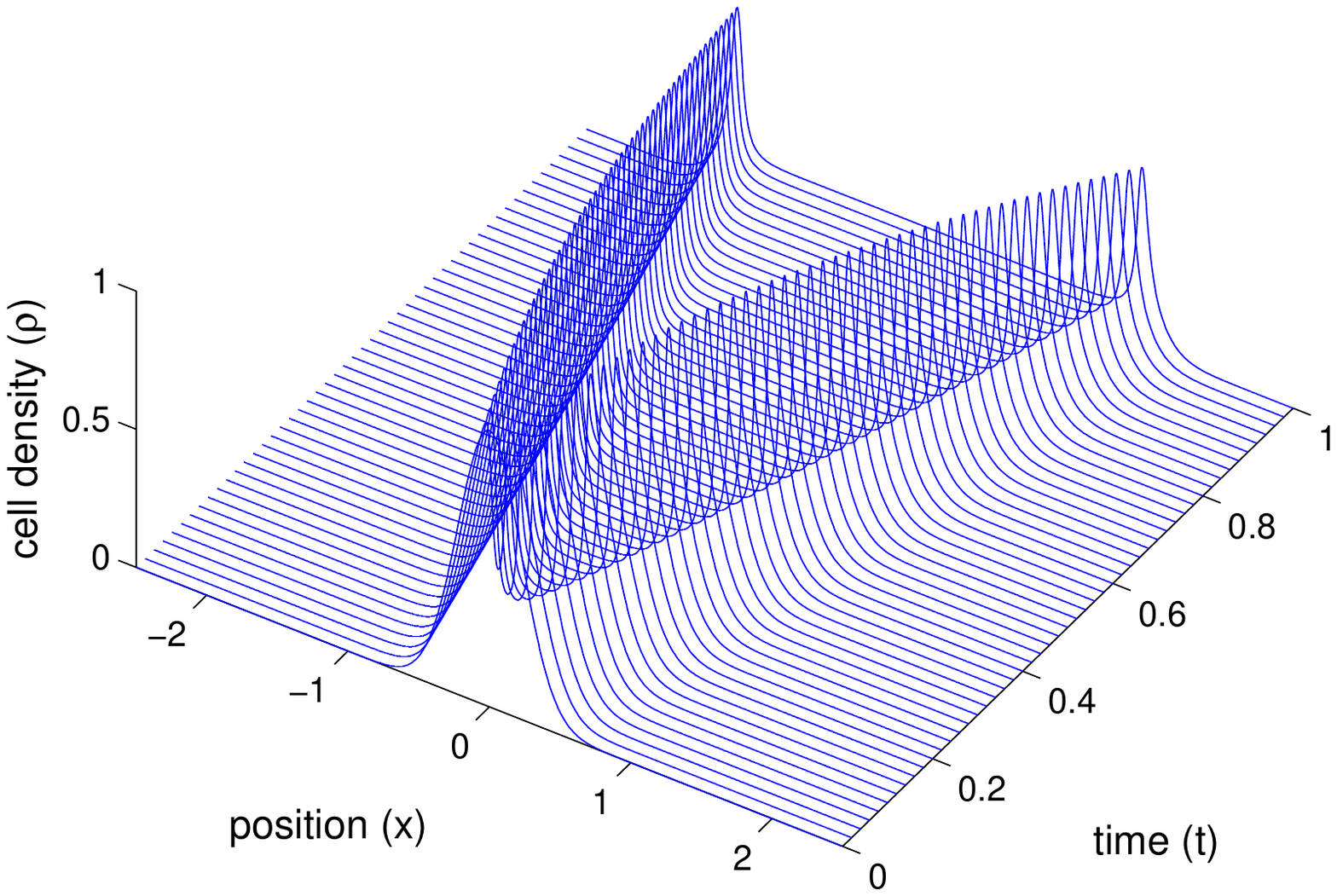}\\
  \caption{Dynamics of the cell density in the repulsive case, i.e. for a non-increasing function $a(x)=-2/\pi \mbox{ Atan}(kx)$. Left: $k=10$; Right: $k=50$.}
  \label{figrepuls1}
  \end{center}
\end{figure}

\begin{figure}[ht!]
\begin{center}
  \includegraphics[width=3.1in]{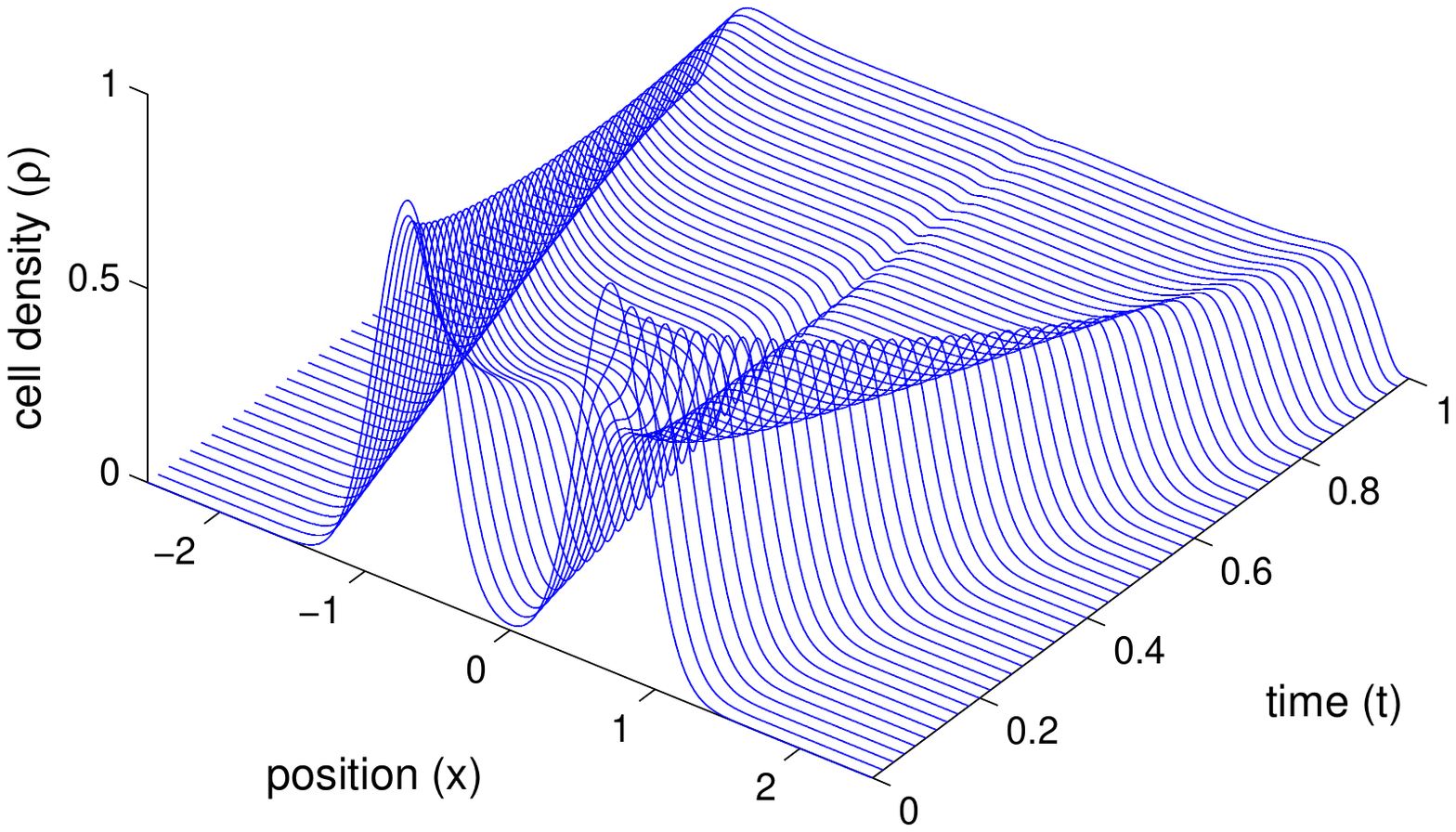}\\
  \caption{Dynamics of the cell density for an initial data given by a sum of 2
    regular bumps in the repulsive case for $a(x)=-2/\pi \mbox{ Atan}(10x)$.}\label{figrepuls2}
  \end{center}
\end{figure}


\end{document}